\documentclass{article}

\usepackage{graphicx} 
\usepackage[T1]{fontenc}
\usepackage[utf8]{inputenc}
\usepackage[english]{babel}
\usepackage{amsmath}
\usepackage{amsfonts}
\usepackage{color}
\usepackage{amsthm}
\usepackage{amssymb}
\usepackage{mathtools}
\usepackage{bm}
\usepackage{tikz}
\usepackage{graphicx}
\usepackage{faktor}
\usepackage{xcolor}
\usepackage{makecell}
\usepackage{caption}
\usepackage{subcaption}
\usepackage{comment}
\usepackage{setspace}
\usepackage[12pt]{moresize}
\usepackage{wasysym}
\usepackage{float}
\usepackage[margin=1in]{geometry}
\usepackage{mathrsfs}
\usepackage{url}
\usepackage{authblk}
\usepackage{esint}
\usepackage{bbm}
\usepackage{hyperref}
\usepackage{enumitem}
\usepackage{natbib}
\usepackage{cleveref}
\usepackage{academicons}
\definecolor{orcidlogocol}{HTML}{A6CE39}
\usepackage{lineno}
\newcommand{\orcidicon}{
  \includegraphics[height=1.8ex]{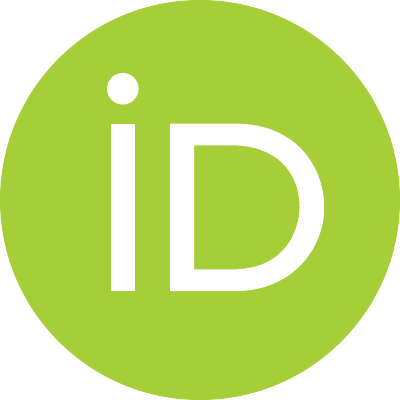}
}
\hypersetup{
colorlinks = true,
urlcolor   = blue,
citecolor  = blue,
linkcolor  = blue,
}

\usepackage{fancyhdr}
\newlength\FHoffset
\fancyheadoffset{\FHoffset}

\newlength\FHleft
\newlength\FHright

\renewcommand{\headrulewidth}{1.0pt} 
\newbox\FHline
\setbox\FHline=\hbox{\hsize=\paperwidth%
  \hspace*{\FHleft}%
  \rule{\dimexpr\headwidth-\FHleft-\FHright\relax}{\headrulewidth}\hspace*{\FHright}%
}

\title{\textbf{
Sharpness of the Osgood Criterion for the Continuity Equation with Divergence-free Vector Fields}}

\author[]{\large 
\textbf{Roberto Colombo}\footnote{EPFL, Station 8,
CH-1015 Lausanne,
Switzerland. \textit{Email:}  \href{mailto:roberto.colombo@epfl.ch}{roberto.colombo@epfl.ch}. 

} \; 
\textbf{\&} \; 
\textbf{Anuj Kumar}\footnote{
Department of Mathematics, University of California Davis, CA 95616, USA. \textit{Email:}  \href{mailto:anku@ucdavis.edu}{anku@ucdavis.edu}. 

\href{https://orcid.org/0000-0002-9203-9177}{\orcidicon orcid.org/0000-0002-9203-9177} 
}}
\date{}

\DeclareMathOperator{\Lip}{Lip}

\DeclareMathOperator{\loc}{loc}
\DeclareMathOperator{\diver}{div}

\DeclareMathOperator{\supp}{supp}

\DeclareMathOperator{\inter}{int}

\newcommand{\eps}{\varepsilon}
\newcommand{\R}{\mathbb{R}}

\newcommand{\N}{\mathbb{N}}

\theoremstyle{plain}
\newtheorem{thm}{Theorem}[section]

\newtheorem*{thm*}{Theorem}
\newtheorem{lem}[thm]{Lemma}
\newtheorem{prop}[thm]{Proposition}

\theoremstyle{remark}

\newcommand{\res}{\mathop{\hbox{\vrule height 7pt width .5pt depth 0pt
			\vrule height .5pt width 6pt depth 0pt}}\nolimits}

\numberwithin{equation}{section}

\date{}

\begin{document}

\maketitle

\begin{abstract}
For any modulus of continuity $\omega$ that fails the Osgood condition, we construct a divergence-free velocity field $v \in C_t C^\omega_x$ for which the associated ODE admits at least two distinct flow maps. In other words, non-uniqueness does not occur merely for a single or even finitely many trajectories, but instead on a set of initial conditions $E$ of positive Lebesgue measure. In fact, the set $E$ has full measure inside a cube where the construction is supported. Moreover, we also construct a divergence-free velocity field $v \in C_{t}C^\omega_x$ for which the associated continuity equation admits two distinct solutions $\mu^1$ and $\mu^2$ which are absolutely continuous with respect to Lebesgue measure for almost every time, and start from the same initial datum $\bar \mu \ll \mathscr{L}^{d}$. 
Our construction introduces two novel ideas: 
\begin{enumerate}[label = (\roman*)]
\item We introduce the notion of \textit{parallelization}, where at each time, the velocity field consists of simultaneous motion across multiple nested spatial scales. This differs from most explicit constructions in the literature on mixing or anomalous dissipation, where the velocity on different scales acts at separate times. This is crucial to cover the whole class of non-Osgood moduli of continuity. 
\item Inspired by the work of Bru\`e, Colombo and Kumar \cite{BCK2024}, we develop a new fixed-point framework that naturally incorporates the parallelization mechanism. This framework allows us to construct anomalous solutions of the continuity equation that belong to $L^1(\mathbb{R}^d)$ a.e. in time. 
\end{enumerate}
 
\end{abstract}


\tableofcontents

\section{Introduction}\label{sec:intro}
A classical problem in analysis is to determine suitable conditions on a Borel vector field $v:[0,T]\times \R^{d}\to \R^{d}$ for which the system of ODEs 
\begin{equation}\tag{ODE}\label{eq:ODE}
    \begin{cases}
        \dot \gamma(t)=v(t,\gamma(t))\qquad &\text{in $(0,T)$},\\
        \gamma(0)=\bar x
    \end{cases}
\end{equation}
is well-posed. Here we say that an absolutely continuous curve $\gamma:[0,T]\to \R^{d}$ solves \eqref{eq:ODE} if
\begin{equation*}
    \gamma(t)=\bar x+\int_{0}^{t}v(s,\gamma(s))ds\qquad \forall t\in [0,T].
\end{equation*}

Existence of solutions to \eqref{eq:ODE} is known under mild assumptions on $v$. An extension of Peano's classical result (\cite{Peano1886, peano1890demonstration}) due to Caratheodory (\cite{caratheodory1918vorlesungen}) ensures that \eqref{eq:ODE} admits at least one solution for every $\bar x\in \R^{d}$ as soon as $v(t, \cdot)$ is continuous for every $t\in [0,T]$ and obeys the following growth condition:
\begin{equation}\label{eq:assumption-growth}
    \int_{0}^{T}\lVert v(t,\cdot)\rVert_{L^{\infty}(\R^{d})}dt<\infty.
\end{equation}

Uniqueness of solutions to \eqref{eq:ODE} generally requires stronger regularity assumptions on the vector field $v$ and the mere spatial continuity is not sufficient. A classical counterexample to uniqueness due to Peano illustrates this fact. Consider a continuous autonomous vector field $v:\R\to \R$ given by $v(x)=\sqrt{|x|}$. Then the function $\gamma_{\alpha}(t):= (t-\alpha)_{+}^{2}/4$ is a solution to \eqref{eq:ODE} with initial condition $\bar{x} = 0$ for every $\alpha\in [0,T]$.

To state the precise requirement to ensure uniqueness, we first introduce the notion of modulus of continuity.  Given a function $\omega:[0,\infty)\to[0,\infty)$ that is continuous, strictly increasing, concave, and such that $\omega(0)=0$, we say that a velocity field $v:[0,T]\times\mathbb{R}^d\to\mathbb{R}^d$ has spatial modulus of continuity $\omega$ if
\begin{align}
[v(t,\cdot)]_{\omega}:= \sup_{\substack{x,y\in \R^{d}\\ x\neq y}}\frac{|v(t,x)-v(t,y)|}{\omega(|x-y|)}
\end{align}
is finite for every $t\in[0,T]$. When $\omega(z)=z$, this corresponds to the Lipschitz case. In this setting, the classical Cauchy--Lipschitz theory (also known as Picard--Lindel\"of Theorem) ensures uniqueness of trajectories for every initial condition, provided that the Lipschitz constant is integrable in time (\cite{picard1893application, lindelof1894application}). The result of Cauchy--Lipschitz can be further generalized to vector fields that are less regular than Lipschitz. In particular, uniqueness of trajectories holds provided that
\begin{equation}\label{eq:Osgood-integral-time}
    \int_{0}^{T}[v(t,\cdot)]_{\omega}dt<\infty,
\end{equation}
where the modulus of continuity $\omega:[0,\infty)\to [0,\infty)$ satisfies the \textit{Osgood condition} \cite{osgood1898beweis} (see also, \cite{Tonelli1925, Montel1926})
\begin{equation}\label{eq:Osgood-condition}
    \int_{0}^{1}\frac{1}{\omega(t)}dt=\infty.
\end{equation}
To see this, one just need to compare two solutions $\gamma_{1},\gamma_{2}$ of \eqref{eq:ODE}
\begin{equation*}
    |\gamma_{1}(t)-\gamma_{2}(t)|\le \int_{0}^{t}[v(s,\cdot)]_{\omega}\omega(|\gamma_{1}(s)-\gamma_{2}(s)|)ds\qquad \forall t\in (0,T)
\end{equation*}
and then use a Gronwall-type argument together with conditions \eqref{eq:Osgood-integral-time} and \eqref{eq:Osgood-condition} to deduce that $|\gamma_{1}-\gamma_{2}|\equiv 0$. 

We say that a modulus of continuity $\omega:[0,\infty)\to[0,\infty)$ which fails to satisfy \eqref{eq:Osgood-condition} is \textit{non-Osgood}.
Letting $\log_{+} z = |\log z|$ denote the absolute value of the logarithm, the functions, $\omega(z) = z \cdot \log_+ z$, $\omega(z) = z \cdot \log_+ z \cdot \log_+ \log_+ z$, are all examples of Osgood moduli. On the other hand, for any $\varepsilon > 0$, the functions $\omega(z) = z\cdot (\log_+ z)^{1 + \varepsilon}$, $\omega(z) = z \cdot \log_+ z \cdot (\log_+ \log_+ z)^{1 + \varepsilon}$ are all examples of non-Osgood moduli. The significance of the Osgood condition \eqref{eq:Osgood-condition} is that it is sharp. In other words, for every non-Osgood modulus $ \omega $, one can construct a vector field $v$, similar to the Peano’s counterexample above, for which the ODE  \eqref{eq:ODE}  admits non-unique solutions for some initial condition $\bar{x}$. 

Although alternative uniqueness criteria for ordinary differential equations are also known, such as Nagumo's condition \cite{nagumo1926hinreichende}, Perron’s criterion \cite{perron1928hinreichende}, and various refinements \cite{agarwal1993uniqueness, constantin2023uniqueness, chu2023nagumo}, the present work focuses on the classical Osgood condition. This is primarily because the Osgood criterion plays a fundamental role beyond the ODE setting and arises naturally in the study of well-posedness for partial differential equations \cite{zhang2014osgood, elgindi2014osgood, la2022regularity, drivas2023propagation, inversi2023lagrangian} as well as stochastic differential equations \cite{leon2014osgood, foondun2021osgood, galeati2022distribution}.

A strictly related problem regards the well-posedness of the Cauchy problem for the continuity (or Liouville) equation
\begin{equation}\tag{PDE}\label{eq:PDE}
    \begin{cases}
        \partial_{t}\mu_{t}+\diver(\mu_{t}v_{t})=0\qquad &\text{in $(0,T)\times \R^{d}$},\\
        \mu_{0}=\bar\mu,
    \end{cases}
\end{equation}
where we use the notation $v_{t}:=v(t,\cdot)$. Here $[0,T]\ni t\mapsto \mu_{t}\in \mathcal{M}(\R^{d})$ is in general a weakly-$*$ continuous curve of finite measures such that 
\begin{equation*}
    \int_{0}^{T}\int_{\R^{d}}|v_{t}|d|\mu_{t}|dt<\infty,
\end{equation*}
and \eqref{eq:PDE} is solved in the sense of distributions:
\begin{equation*}
    \int_{\R^{d}}\varphi d\mu_{t}=\int_{\R^{d}}\varphi d\bar\mu+\int_{0}^{t}\int_{\R^{d}}\nabla \varphi\cdot v_{t}d\mu_{t}dt\qquad \forall t\in [0,T],\quad \forall \varphi \in C^{\infty}_{c}(\R^{d}).
\end{equation*}
Note that \eqref{eq:PDE} reduces to \eqref{eq:ODE} when we consider Dirac measure solutions $\mu_{t}=\delta_{\gamma(t)}$. Conversely, bundling trajectories we can construct solutions of \eqref{eq:PDE} by means of the so called superposition principle. For $\bar\mu$-a.e. $\bar x\in \R^{d}$, take $\eta_{\bar x}\in \mathscr{P}(C([0,T];\R^{d}))$ a probability in the space of paths which is concentrated on the set of solutions to \eqref{eq:ODE}. Then, calling $e_{t}:C([0,T];\R^{d})\to \R^{d}$ the evaluation map $e_{t}(\gamma)=\gamma(t)$, for every $t\in [0,T]$, the measure defined by the disintegration
\begin{equation*}
    \mu_{t}= (e_{t})_{\#}\eta_{\bar x}\otimes d\bar\mu (\bar x)\qquad \forall t\in [0,T]
\end{equation*}
is a solution to \eqref{eq:PDE}, $\#$ denoting the push forward operator. In fact, Ambrosio \cite{Ambrosio04} proved that all non-negative solutions of \eqref{eq:PDE} are constructed in this way (see also \cite[Theorem 3.4]{ambrosio2014continuity}). 

The characterization provided by Ambrosio's superposition principle is very useful to pass well-posedness results from the \eqref{eq:ODE} level to the \eqref{eq:PDE} level. For example, under conditions \eqref{eq:assumption-growth}, \eqref{eq:Osgood-integral-time} and \eqref{eq:Osgood-condition}, we know that there exists a unique solution $t\mapsto X(t,\bar x)$ of \eqref{eq:ODE} for any starting point $\bar x$, hence we infer that for every $\bar\mu \in \mathcal{M}_{+}(\R^{d})$ the unique solution of \eqref{eq:PDE} in the class of non-negative measures is given by
\begin{equation}\label{eq:push-forward-flow}
    \mu_{t}=(X_{t})_{\#}\bar\mu\qquad \forall t\in [0,T].
\end{equation}
Here $X:[0,T]\times \R^{d}\to \R^{d}$ is the flow map associated to the vector field $v$. In \cite{ambrosio2008uniqueness}, by means of Smirnov's decomposition Theorem for $1$-currents \cite{smirnov1994decomposition}, Ambrosio and Bernard proved that under conditions \eqref{eq:assumption-growth}, \eqref{eq:Osgood-integral-time} and \eqref{eq:Osgood-condition}, the unique solution of \eqref{eq:PDE} is given by \eqref{eq:push-forward-flow}, even for general signed measures $\bar\mu\in \mathcal{M}(\R^{d})$.
In conclusion, the Osgood condition \eqref{eq:Osgood-condition} is precisely the sharp assumption needed on the spatial modulus of continuity of $v$ in order to guarantee well-posedness of both \eqref{eq:ODE} and \eqref{eq:PDE}. We refer to \cite{fjordholm2025transport, inversi2023lagrangian} for recent results extending the Osgood criterion to general Transport equations and systems of nonlocal continuity equations.  

In fluid mechanics, an important class is that of incompressible flows, which are generated by vector fields $v:[0,T]\times \R^{d}\to \R^{d}$ satisfying the divergence-free condition 
\begin{equation}\label{eq:div-free-condition}
    \diver v(t,\cdot)=0\qquad \forall t\in [0,T].
\end{equation}
Under \eqref{eq:div-free-condition}, one can equivalently rewrite \eqref{eq:PDE} as a Transport equation
\begin{equation}
\tag{T-Eq} \label{eq: transport}
 \begin{cases}
        \partial_{t}\mu_{t}+v_{t}\cdot\nabla\mu_{t}=0\qquad &\text{in $(0,T)\times \R^{d}$},\\
        \mu_{0}=\bar\mu.
    \end{cases}
\end{equation}
In this context, Osgood-type conditions play an important role in questions of regularity and uniqueness \cite{zhang2014osgood, elgindi2014osgood, la2022regularity, drivas2023propagation}, and establishing whether the Osgood criterion is sharp for divergence-free vector fields, both at the level of ODE and PDE, is therefore of significant interest. While constructing non-unique solutions to the ODE for any non-Osgood modulus of continuity and some initial condition $\bar{x}$ is straightforward, the non-uniqueness of a single trajectory, or even finitely many trajectories, is not sufficient to guarantee the existence of two distinct flow maps, that is, flow maps that differ on a set of positive measure. Consequently, for a given non-Osgood modulus of continuity, a more subtle and important question is whether non-uniqueness of trajectories can occur on a set $E\subset \R^{d}$ of initial conditions with positive Lebesgue measure, $\mathscr{L}^{d}(E) > 0$, for a divergence-free vector field. 

In this direction, Liss and Elgindi \cite{elgindi2024norm} constructed a divergence-free vector field with modulus of continuity $\omega(z) = z |\log z|^2$ for which trajectories are non-unique on a set of positive measure. They established this result in the context of anomalous dissipation in the scalar diffusion equation. Their construction is based on alternating shear flows with frequency of the flow exploding to infinity at a finite time $t = T$, which corresponds to the onset of non-uniqueness. However, beyond the modulus of continuity $\omega(z) = z |\log z|^2$, and in particular for general non-Osgood moduli, the construction of a divergence-free vector field for which trajectories exhibit non-uniqueness on a set of positive measure remained open (see the discussion in \cite{elgindi2024norm}[Section 3.1]).

The main result of this paper gives an answer to this question: for every non-Osgood modulus of continuity we construct a divergence-free vector field for which trajectories are non-unique in a set of positive measure. This establishes the sharpness of the Osgood criterion for divergence-free vector fields when the uniqueness of flow maps is concerned. In fact, our result is stronger. From a PDE perspective, we obtain non-unique solutions in $L^{\infty}_{t}L^{1}_{x}$, the class of measures which are absolutely continuous with respect to Lebesgue measure for almost every time. Denoting by $C^{\omega}(\R^{d})$ the space of $\omega$-continuous vector fields (See \Cref{subsec:MOC}), the precise statement of the result is as follows:
\begin{thm}\label{thm:main}
    Let $d\ge 2$ and $T > 0$. Given any non-Osgood modulus of continuity $\omega:[0,\infty)\to [0,\infty)$, there exists a divergence-free velocity field $v \in C([0, T]; C^\omega(\mathbb{R}^d))$ such that there are two distinct solutions $\mu^1, \mu^2$ of the equation \eqref{eq:PDE} in the class $C_{w^\ast}([0, T]; \mathcal{M}_{+}(\mathbb{R}^d))$ starting from the same initial condition $\bar\mu\ll \mathscr{L}^{d}$. Moreover, these solutions satisfy $\mu^1_t \ll \mathscr{L}^{d}$ and $\mu^2_t \ll \mathscr{L}^{d}$ for a.e. $t \in [0, T]$. In particular, there exist a Borel set $E\subset \R^{d}$ with $\mathscr{L}^{d}(E)>0$ such that for every $\bar x\in E$ there are at least two distinct solutions of \eqref{eq:ODE} starting at $\bar x$.
\end{thm}
As the velocity field is divergence-free, our result can also be phrased in terms of the transport equation:
we construct examples of incompressible velocity fields (corresponding to a given non-Osgood modulus) for which the transport equation \eqref{eq: transport} exhibits non-unique solutions in the space $L^\infty_t L^1_x$ starting from the same initial condition $\bar \mu\ll\mathscr{L}^{d}$.

One of the main novelties of our construction is that the velocity field exhibits superimposed (or parallel) motions at several nested length scales simultaneously. This idea, which we call \textit{parallelization}, is fundamentally different from many existing explicit constructions concerning mixing \cite{AlbertiCrippaMazzucato19, YaoZlatos17}, anomalous dissipation \cite{drivas22anomdissp, colombo2023anomalous, elgindi2024norm} or examples of non-uniqueness \cite{depauw2003non, de2022smoothing, K2023, BCK2024, huysmans2025non} in the transport or scalar diffusion equations, where all the action takes place at a given length scale at a given time.

Before proving Theorem \ref{thm:main}, we first construct a velocity field (for a prescribed non-Osgood modulus of continuity) for which the trajectories are non-unique for almost every starting point in a cube where the field is supported, but the associated PDE solution is not absolutely continuous with respect to Lebesgue measure, except at a single time (See \Cref{prop: nonunique traj}). The purpose of this preliminary construction is to introduce the idea of parallelization, which can be viewed as a parallelized version of the construction given in Kumar \cite{K2023}.
Once this base construction is established, we combine it with a fixed-point argument, in the spirit of Bru\`e, Colombo and Kumar \cite{BCK2024}. The introduction of parallelization leads to new technical challenges while using the fixed-point approach, which do not arise in \cite{BCK2024}. We overcome these difficulties and thereby, complete the proof of Theorem \ref{thm:main}.

We emphasize that parallelization is an important ingredient in our construction of a divergence-free velocity field that leads to non-uniqueness of the flow map for a given non-Osgood modulus. Without the concept of parallelization, the best we have been able to achieve with similar constructions is non-uniqueness with a modulus of continuity $\omega(z) = z |\log z|^2$, which coincides with the result of Liss and Elgindi \cite{elgindi2024norm}. As suggested by the heuristic argument outlined in \Cref{subsec: overview const traj}, some form of parallelized motion may be necessary to go beyond this threshold and reach the sharp result of general non-Osgood moduli. From this perspective, we believe that the concept of parallelization introduced in this paper may prove fruitful in other settings as well. In particular, it has the potential to be combined with convex integration technique \cite{de2009euler,DeLellisSzekelyhidi13}, for instance in the context of transport equations \cite{modena2018non, modena2020convex, BCDL2021, colombo2025convex}, or of two-dimensional incompressible Euler equations with vorticity in $L^p$ \cite{buck2024non, brue2024flexibility, buck2024compactly}, possibly leading to further progress on non-uniqueness results for fluid PDEs.

The paper is organized as follows. In Section \ref{sec:preliminaries}, we introduce some important notations, and we prove a few auxiliary lemmas. Section \ref{sec:proof weak} presents the base construction introducing the concept of parallelization and leading to non-uniqueness of flow maps, as mentioned above. In Section \ref{sec: proof main thm}, we develop the main construction, which combines parallelization with a fixed point argument, and allows to conclude the proof of Theorem \ref{thm:main}.

\section{Notations and basic lemmas}\label{sec:preliminaries}
In this section we set some notations and lemmas that will be useful in the constructions of \Cref{sec:proof weak} and \Cref{sec: proof main thm}. In \Cref{subsec:Cantor} we introduce the notations related to the Cantor set, which represents the basic geometric background of later constructions. Next, in \Cref{subsec:BB} we construct the building blocks, which are compactly supported incompressible vector fields used to rigidly translate cubes of different sizes. Finally, \Cref{subsec:MOC} contains a few lemmas related to moduli of continuity.  
\subsection{The Cantor set}\label{subsec:Cantor}
The preliminary velocity field of Section \ref{sec:proof weak} is designed according to a dyadic Cantor set construction in the spirit of \cite{K2023}. Let the dyadic alphabet be given by
\begin{equation*}
    \mathfrak{S}:=\left\{-1,1\right\}^{d}.
\end{equation*}
For $\sigma \in \mathfrak{S}^n$, $n\ge 1$, we write $\sigma = (\sigma_1,  \dots ,\sigma_n)$ to denote the components of $\sigma$. Given $n\le n'$ and $\sigma \in \mathfrak{S}^{n}$, $\tilde\sigma \in \mathfrak{S}^{n'}$, we write $\sigma \subset \tilde\sigma$ if $\sigma_i = \tilde\sigma_{i}$ for all $i \in \{1, \dots, n\}.$ Finally, for every $\sigma \in \mathfrak{S}^n$, $n \geq 2$, we use $\sigma^\prime$ to denote the unique element in $\mathfrak{S}^{n-1}$ such that $\sigma^\prime \subset \sigma$. 

Given any sequence of positive numbers $\eta=\{\eta_{n}\}_{n\ge 1}\in (0,\infty)^{\mathbb{N}}$, we denote by $\nu^{\eta}=\{\nu^{\eta}_{n}\}_{n\ge 0}\in (0,\infty)^{\mathbb{Z}_{\geq0}}$ the sequence of cumulative sums, i.e.
\begin{equation} \label{eq: nu from eta}
    \nu^{\eta}_{0}:=0,\qquad \nu^{\eta}_{n}:=\overunderset{n}{j=1}{\sum}\eta_{j}\quad \forall n\ge 1.
\end{equation}
We then denote by $\{\ell^{\eta}_{n}\}_{n\ge 0}$ the sequence of positive lengths obtained as follows:
\begin{equation*}
    \ell_{n}^{\eta}:=2^{-n-\nu^{\eta}_{n}}\qquad \forall n\ge 0.
\end{equation*}
For every number $\eps \in (0,\infty)$ and every symbol $\sigma \in \mathfrak{S}$, let $T^{\eps}_{\sigma}:\mathbb{R}^{d}\rightarrow \mathbb{R}^{d}$ be the similarity transformation
\begin{equation*}
    T^{\eps}_{\sigma}(x):= 2^{-1-\eps} x+ \frac{\sigma}{4},
\end{equation*}
and consider the unit cube 
\begin{equation*}
    Q:=\left[-\frac{1}{2},\frac{1}{2}\right]^{d}.
\end{equation*}
The $n$th Cantor generation set associated to the sequence $\eta$ is defined as
\begin{equation*}
   \mathscr{C}_{n}^{\eta}:= \underset{\sigma_{1}\in \mathfrak{S}}{\bigcup}\dots \underset{\sigma_{n}\in \mathfrak{S}}{\bigcup}\,T^{\eta_{1}}_{\sigma_{1}}\circ\dots\circ T^{\eta_{n}}_{\sigma_{n}}(Q)\qquad \forall n\ge 1.
\end{equation*}
Each Cantor generation $\mathscr{C}_{n}^{\eta}$ is the union of $2^{dn}$ pairwise disjoint closed cubes, each of side length $\ell^{\eta}_{n}$. Note that the Cantor generations are nested, and all contained inside the unit cube: $\mathscr{C}^{\eta}_{n+1}\subset \mathscr{C}^{\eta}_{n}\subset Q$ for every $n\ge 1$. By taking the intersection of all the Cantor generations, we obtain the non-empty compact set 
\begin{equation} \label{eq: cantor set}
    \mathscr{C}^{\eta}:= \bigcap_{n\ge 1}\mathscr{C}^{\eta}_{n}\subset Q
\end{equation}
which we call the Cantor set associated with the sequence $\eta$. The measure theoretic properties of the Cantor set $\mathscr{C}^{\eta}$ are clearly influenced by the choice of the sequence $\eta=\{\eta_{n}\}_{n\ge 1}\subset (0,\infty)$. For instance, one can compute the Haussdorf dimension of $\mathscr{C}^{\eta}$ in terms of the asymptotic behavior of $\eta_{n}$ as $n\to \infty$. For the sequel, we only note that $\mathscr{L}^{d}(\mathscr{C}^{\eta})=0$ whenever $\sum_{n\ge 1}\eta_{n}=\infty$.    

The construction, as describe above, provides a natural way to assign to each point $x\in \mathscr{C}^{\eta}$ a target point $S^{\eta}(x)\in Q$, obtained by first associating $x$ with the unique sequence $\{\sigma_{n}(x)\}_{n\ge 1}\in \mathfrak{S}^{\N}$  of symbols from the dyadic alphabet for which 
\begin{equation*}
    \left\{x\right\}=\underset{n\ge 1}{\bigcap}\,T^{\eta_{1}}_{\sigma_{1}(x)}\circ \dots \circ T^{\eta_{n}}_{\sigma_{n}(x)}(Q),
\end{equation*}
and then using this sequence for the binary representation of $x$ inside $Q$:
\begin{equation} \label{eq: S eta}
    S^{\eta}: \mathscr{C}^{\eta}\rightarrow Q,\qquad  S^{\eta}(x):= \underset{n\ge 1}{\sum}2^{-n-1}\sigma_{n}(x).
\end{equation}
One can check that the map $S^{\eta}$ is surjective and essentially injective: almost every point $x\in Q$ has a unique pre-image $S^{-1}(x)$ (except the points for which $2^m x \in \mathbb{Z}^{d}$ for some $m \in \mathbb{N}$). 

We conclude this part by defining a few points in the cube $Q$ that will be used later in this paper: 
\begin{equation} \label{def: p and s}
    p^{\eta}_{n,\sigma}:= \sum_{j=1}^{n}\frac{\ell^{\eta}_{j-1}}{4}\sigma_{j},\qquad s_{n,\sigma}:= S^{\eta}(p_{n,\sigma}^{\eta})= \overunderset{n}{j=1}{\sum}2^{-j-1}\sigma_{j}\qquad \qquad \forall \; n\ge 1,\quad \forall \; \sigma\in \mathfrak{S}^{n}. 
\end{equation}
Observe that $p^{\eta}_{n,\sigma}$ and $s_{n,\sigma}$ coincide, respectively, with the center of the cube of the $n$th Cantor generation and of the $n$th dyadic partition associated to the $n$-uple $\sigma$. 

\subsection{The building block}\label{subsec:BB}
The following lemma provides the construction of the fundamental building block vector field, which will be used to translate cubes of various sizes in Section \ref{sec:proof weak} and Section \ref{sec: proof main thm}. This type of building block was first used in \cite{BCDL2021}, and later considered in several other papers on non-uniqueness for transport equations (\cite{K2023}, \cite{BCK2024}, \cite{colombo2025convex}). We include here the construction for the reader's convenience. 

\begin{lem}\label{lem-Building-block}
    Given $e\in \R^{d}$,
   there is a smooth vector field $u^{e}\in C^{\infty}(\R^{d};\R^{d})$ such that
    $$\diver u^{e}=0,\qquad u^{e}(x)=0\quad \forall x\in\R^{d}\setminus \frac{3}{4}Q,\qquad u^{e}(x)=e \quad \forall x\in \frac{1}{2}Q,\qquad \lVert u^{e}\rVert_{C^{1}(\R^{d})}\lesssim_{d} |e|.$$ 

\end{lem}
\begin{proof}
 Up to a multiplication by a constant we may assume that $|e|=1$.
 Let $\tilde{e}\in \R^{d}$ be such that $|\tilde e|=1$ and $e\cdot \tilde{e}=0$. Let $\psi\in C^{\infty}_{c}(\R^{d})$ be such that 
 \begin{equation}\label{eq:properties-stream-function-BB}
     \psi(x)=0 \quad \forall x\in \R^{d}\setminus \frac{3}{4}Q,\qquad \psi(x)=\tilde{e}\cdot x\quad \forall x\in \frac{1}{2}Q.
 \end{equation}
 We define $u^{e}:= \partial_{\tilde{e}}\psi e-\partial_{e}\psi \tilde e\in C^{\infty}_{c}(\R^{d})$. Note that $\diver u^{e}= \partial_{e}\partial_{\tilde e}\psi-\partial_{\tilde e}\partial_{e}\psi =0$ in $\R^{d}$. Moreover, by \eqref{eq:properties-stream-function-BB},
 \begin{equation*}
     u^{e}(x)=0\quad \forall x\in \R^{d}\setminus \frac{3}{4}Q,\qquad u^{e}(x)= e\quad \forall x\in \frac{1}{2}Q,
 \end{equation*}
 as desired. Finally, $\lVert u^{e}\rVert_{C_{1}(\R^{d})}\lesssim \lVert \psi\rVert_{C^{2}(\R^{d})}\lesssim_{d} 1$. 
\end{proof}
The prototypical use of $u^{e}$ in the sequel is as follows: the time dependent field $v:[0,1]\times \R^{d}\to \R^{d}$ such that $v(t,x):=u^{e}(x-te)$ rigidly translates the cube $Q/2$ to $e+Q/2$ in a unity of time. That is, the flow map $X:[0,1]\times \R^{d}\to \R^{d}$ associated to $v$ is such that $X(t,Q/2)=te+Q/2$ for every $t\in [0,1]$, or $\mu_{t}:= \mathbbm{1}_{te+Q/2}\mathscr{L}^{d}$ solves the continuity equation \eqref{eq:PDE} with velocity field $v$. Of course, suitable time reparametrizations and space rescalings of the basic example above allow to translate cubes of arbitrary sizes in arbitrary intervals of time. 

\subsection{Moduli of continuity}\label{subsec:MOC}
We recall that a modulus of continuity is a continuous function $\omega:[0,\infty)\to [0,\infty)$ which is strictly increasing, concave and such that $\omega(0)=0$. The modulus $\omega$ is Osgood if condition \eqref{eq:Osgood-condition} holds, and non-Osgood otherwise. 
Let $C^{\omega}(\R^{d})$ be the space of functions $u:\R^{d}\to \R^{d}$ such that 
\begin{equation*}
    \lVert u\rVert_{C^{\omega}(\R^{d})}:= \lVert u\rVert_{L^{\infty}(\R^{d})}+[u]_{C^{\omega}(\R^{d})}<\infty,
\end{equation*}
the semi-norm $[u]_{C^{\omega}(\R^{d})}$ being defined as
\begin{align}\label{eq:def-seminorm-omega}
[u]_{C^{\omega}(\R^{d})}:= \sup_{\substack{x,y\in \R^{d}\\ x\neq y}}\frac{|u(x)-u(y)|}{\omega(|x-y|)}.
\end{align}

We start with the following interpolation lemma, which reduces the boundedness of the semi-norm $[u]_{C^{\omega}(\R^{d})}$ to suitable bounds on the uniform norm $\lVert u\rVert _{L^{\infty}(\R^{d})}$ and on the Lipschitz semi-norm $\lVert \nabla u\rVert_{L^{\infty}(\R^{d})}$.

\begin{lem}\label{lem:interpolation}
    Let $\omega:[0,\infty)\to [0,\infty)$ by any modulus of continuity. Let $u:\R^{d}\to \R^{d}$ be such that
    \begin{equation*}
        \lVert u\rVert_{L^{\infty}(\R^{d})}\le C\omega(r),\qquad \lVert \nabla u\rVert_{L^{\infty}(\R^{d})}\le C\frac{\omega(r)}{r},
    \end{equation*}
    for some $r,C>0$. 
    Then $[u]_{C^{\omega}(\R^{d})}\le 2C$.
\end{lem}
\begin{proof}
    Let $x,y \in \R^{d}$ be such that $x\neq y$. If $|x-y|\le r$, The Lipschitz bound and the concavity of $\omega$ give
    \begin{equation*}
        |u(x)-u(y)|\le C\frac{\omega(r)}{r}|x-y|\le C\omega(|x-y|).
    \end{equation*}
    If instead $|x-y|>r$, we can use the $L^{\infty}$ bound and the monotonicity of $\omega$ to get
    \begin{equation*}
        |u(x)-u(y)|\le 2C\omega(r)\le 2C\omega(|x-y|).
    \end{equation*}
    Hence, in all cases $|u(x)-u(y)|\le 2C\omega(|x-y|)$, as desired. 
\end{proof}

Next we show that for any non-Osgood modulus $\omega$ there exists another non-Osgood modulus $\tilde{\omega}$ which grows slower than $\omega$ at zero, that is $\omega(r) / \tilde{\omega}(r) \to \infty$ as $r \to 0^+$. Basing the constructions of \Cref{sec:proof weak} and \Cref{sec: proof main thm} on the auxiliary more restrictive modulus of continuity $\tilde{\omega}$, will enable us to obtain the desired $C_{t}C_{x}^{\omega}$ continuity of the velocity field. We first need a lemma about infinite series.
\begin{lem} \label{lem: series}
Let $\{b_n\}_{n\ge 1}$ be a non-increasing sequence of positive numbers such that $\sum_{n\ge 1}b_{n}<\infty$. There exists a non-decreasing sequence of positive numbers $\{a_n\}_{n\ge 1}$ such that $a_{1}=1$, $a_n \to \infty$ as $n \to \infty$, $\sum_{n\ge 1}a_{n+1} b_n<\infty$, and 
\begin{align} \label{eq: extra const}
\left(\frac{2}{a_n} - \frac{1}{a_{n+1}}\right) b_{n-1} \geq \left(\frac{2}{a_{n-1}} - \frac{1}{a_{n}}\right) b_{n} \qquad \forall n\ge 2.
\end{align}
\end{lem}
\begin{proof}
As we are looking for a non-decreasing sequence $\{a_n\}_{n\ge 1}$, instead of \eqref{eq: extra const} it is sufficient to require 
\begin{align}\label{eq:reformulation-constraint-lemma-moc}
\frac{b_{n-1}}{a_n} \geq \left(\frac{2}{a_{n-1}} - \frac{1}{a_{n}}\right) b_{n}, \qquad \text{or} \qquad
\frac{a_n}{a_{n-1}} \leq \frac{b_n + b_{n-1}}{2b_n}\qquad \forall n\ge 2. 
\end{align}
Let $\{\bar{a}_n\}_{n\ge 1}$ be a non-decreasing sequence (not necessarily satisfying the extra constraint \eqref{eq:reformulation-constraint-lemma-moc}) such that $\bar{a}_n \to \infty$ as $n\to \infty$, and $\sum_{n\ge 1} \bar{a}_{n+1} b_n<\infty$. For example, define indices $1 = n_0 = n_1 < n_2 < \dots$ such that $\sum_{n\ge n_{k}} b_n < 2^{-k}$ for all $k \geq 1$ and then choose $\bar{a}_n = k$ for all $n_{k-1}+1 \leq n \le n_k$ and all $k\ge 1$. Having chosen $\bar{a}_{n}$, we recursively define the required sequence as follows: 
\begin{align*}
a_{1}:= 1,\qquad a_{n+1} := \min\left\{\frac{b_{n+1} + b_{n}}{2b_{n+1}} a_{n}, \bar{a}_{n+1}\right\} \quad \forall n\ge 1.
\end{align*}
Clearly, $\{a_n\}_{n\ge 1}$ satisfies the constraint \eqref{eq:reformulation-constraint-lemma-moc}. In addition, since $b_{n}$ is non-increasing and $\bar a_{n}$ is non-decreasing, we deduce that $a_{n}$ is non-decreasing as well. The only property that is left to be checked is $a_n \to \infty$ as $n \to \infty$. Let $N = \sup\{n \in \mathbb{N} \, | \, a_n > \bar{a}_n/2\}$. If $N = \infty$, then there is nothing to prove. Otherwise, 
\begin{align*}
\frac{a_n}{a_{n-1}} = \frac{b_n + b_{n-1}}{2b_n}  \quad \forall n > N \quad \implies \quad \frac{a_{k+N}}{a_N} = \prod_{n = N+1}^{N+k} \frac{b_n + b_{n-1}}{2b_n} \geq \prod_{n = N+1}^{N+k} \left(\frac{b_{n-1}}{b_n}\right)^{\frac{1}{2}} = \left(\frac{b_N}{b_{N+k}}\right)^{\frac{1}{2}}.
\end{align*}
Since $b_{N+k} \to 0$ as $k \to \infty$, this completes the proof.
\end{proof}

We are now in the position to prove the following
\begin{lem}\label{lem:auxiliary-moc}
    Let $\omega:[0,\infty)\to [0,\infty)$ be a non-Osgood modulus of continuity. There exists another non-Osgood modulus of continuity $\tilde{\omega}:[0,\infty)\to [0,\infty)$ such that $\tilde{\omega}\le \omega$, and
    \begin{equation*}
    \frac{\omega(r)}{\tilde{\omega}(r)}\to \infty\qquad \text{as $r\to 0^{+}$}.
\end{equation*}
\end{lem}
\begin{proof}
    We call $\Omega:[0,\infty)\to [0,\infty)$ the strictly increasing continuous function defined by
    \begin{equation*}
        \Omega(r):= \int_{0}^{r}\frac{1}{\omega(t)}dt\qquad \forall r\ge 0.
    \end{equation*}
    Note that $\Omega$ is well-defined thanks to the non-Osgood condition on $\omega$, and $\Omega(0)=0$.
    Let $\{r_{n}\}_{n\ge 1}\subset (0,1]$ be the decreasing infinitesimal sequence uniquely identified by the relation
    \begin{equation}\label{eq:def-rn-lemma-moc}
        \omega(r_{n})= 2^{-n+1}\omega(1)\qquad \forall n\ge 1.
    \end{equation}
    The concavity of $\omega$ in the form of monotonicity of difference quotients together with \eqref{eq:def-rn-lemma-moc} imply that the sequence $b_{n}:=2^{n}(r_{n}-r_{n+1})$, $n\ge 1$, is non-increasing. We may then apply  \Cref{lem: series} and find a non-decreasing sequence of positive numbers $\{a_{n}\}_{n\ge 1}$ such that $a_{1}=1$, $a_n \to \infty$ as $n \to \infty$, and
    \begin{gather}
        \sum_{n\ge 1} 2^n (r_n - r_{n+1}) a_{n+1} < \infty, \label{eq:cond-1-lemma-moc}\\
        \left(\frac{2}{a_n} - \frac{1}{a_{n+1}}\right) (r_{n-1}-r_{n}) \geq 2\left(\frac{2}{a_{n-1}} - \frac{1}{a_{n}}\right) (r_{n}-r_{n+1}) \qquad \forall n\ge 2.\label{eq:cond-2-lemma-moc}
    \end{gather}
    The modulus of continuity $\tilde{\omega}$ is then defined as the following piecewise affine function:
    \begin{equation*}
        \tilde{\omega}(t):= \begin{cases}
            \omega(t)\qquad &\forall t\ge 1,\\
            \frac{r_{n}-t}{r_{n}-r_{n+1}}\frac{\omega(r_{n+1})}{a_{n+1}}+\frac{t-r_{n+1}}{r_{n}-r_{n+1}}\frac{\omega(r_{n})}{a_{n}}\qquad &\forall t\in [r_{n+1},r_{n}),\quad \forall n\ge 0,\\
            0\qquad & t=0
        \end{cases}
    \end{equation*}
    $\tilde{\omega}$ is continuous and strictly increasing in $[0,\infty)$. Let us prove that it is non-Osgood. Since $\tilde{\omega}(t)\ge \tilde{\omega}(r_{n+1})=\omega(r_{n+1})/a_{n+1}=\omega(1)/2^{n}a_{n+1}$ in $[r_{n+1},r_{n})$, by \eqref{eq:cond-1-lemma-moc} we have
    \begin{align*}
        \int_{0}^{1}\frac{1}{\tilde{\omega}(t)}dt= \sum_{n\ge 1}\int_{r_{n+1}}^{r_{n}}\frac{1}{\tilde{\omega}(t)}dt\le \frac{1}{\omega(1)}\sum_{n\ge 1}2^{n}(r_{n}-r_{n+1})a_{n+1}<\infty.
    \end{align*}
    Finally, to prove that $\tilde{\omega}$ is concave it is sufficient to show that the slopes of the affine functions defining $\tilde{\omega}$ in $[r_{n+1},r_{n})$ are non-decreasing in $n$:
    \begin{equation*}
        \frac{\tilde{\omega}(r_{n})-\tilde{\omega}(r_{n+1})}{r_{n}-r_{n+1}}\ge \frac{\tilde{\omega}(r_{n-1})-\tilde{\omega}(r_{n})}{r_{n-1}-r_{n}}\qquad \forall n\ge 2,
    \end{equation*}
    which corresponds exactly to condition \eqref{eq:cond-2-lemma-moc}, after substituting $\tilde{\omega}(r_{n})= \omega(1)/2^{n-1}a_{n}$ for all $n\ge 1$. 
\end{proof}

\section{Non-uniqueness of trajectories in a set of full measure}\label{sec:proof weak}
Before going to the proof of \Cref{thm:main} in \Cref{sec: proof main thm}, we wish to prove here the following simpler result: 
\begin{prop} \label{prop: nonunique traj}
    Let $d\ge 2$ and $\omega:[0,\infty)\to [0,\infty)$ be any non-Osgood modulus of continuity. For every $T>0$, there exists a divergence-free velocity field $v \in C([0, T]; C^\omega(\mathbb{R}^d))$ such that $\supp v(t, \cdot) \subseteq Q$ for all $t\in [0,T]$, and for which the \eqref{eq:ODE} admits at least two distinct trajectories for almost every starting point in $Q$.
\end{prop}

The proof of \Cref{prop: nonunique traj} is based on a time reversal argument applied to a velocity field $b$ whose associated unique flow map $X$ takes a Cantor set onto the entire cube $Q$ in finite time. As such, we begin by fixing the Cantor set used in the construction of $b$. We choose the sequence $\eta=\{\eta_n\}_{n\ge 1}\in(0,\infty)^{\N}$ defined by $\eta_n=1$ for all $n\ge1$. From equation $\eqref{eq: nu from eta}$, this choice yields $\nu_n^\eta = n$. Let $\mathscr{C}^\eta$ denote the corresponding Cantor set defined in $\eqref{eq: cantor set}$. With this choice we have in particular $\mathscr{L}^d(\mathscr{C}^\eta)=0$.
The following proposition outlines the  properties of the vector field $b$:
\begin{prop} \label{prop: nonunique traj rev}
Let $d\ge 2$ and $\omega:[0,\infty)\to [0,\infty)$ be any non-Osgood modulus of continuity. For every $T > 0$, there exists a divergence-free velocity field $b \in C([0, T]; C^\omega(\mathbb{R}^d))$ such that $\supp b(t, \cdot) \subseteq Q$ for all $t\in [0,T]$, it admits unique trajectories, and the flow map $X:[0, T] \times \mathbb{R}^d \to \mathbb{R}^d$ has the property that $X(T, x) = S^\eta(x)$ for all $x \in \mathscr{C}^\eta$, where the $S^\eta$ is defined in \eqref{eq: S eta}. 
\end{prop}

\begin{proof}[Proof of \Cref{prop: nonunique traj}]
With \Cref{prop: nonunique traj rev} in hand, the proof of \Cref{prop: nonunique traj} is straightforward. For brevity, we drop $\eta$ from the superscript in the notation of the mapping $S^\eta$ and the Cantor set $\mathscr{C}^\eta$. We define $v$ as
    \begin{equation*}
        v(t,x):=-b(T-t, x)\qquad \forall (t,x)\in [0,T]\times \R^{d}.
    \end{equation*}
    Taking into account the properties of $b$, one thing that is remained to be shown is that $v$ has non-unique trajectories on a full measure set of initial conditions in $Q$. On the one hand, by the essential invertibility of the target map $S$, for almost every $x\in \R^{d}$, we can define the inverse flow map
    \begin{equation*}
        Y(t,x):= X(T-t, S^{-1}(x))\qquad \forall t\in [0,T],
    \end{equation*}
    such that $Y(\cdot ,x)$ is a solution of \eqref{eq:ODE} with vector field $v$ and initial condition $x$. On the other hand, by the incompressibility of $v$ and the fact that $v_{t}$ vanishes in $\R^{d}\setminus Q$ for every $t\in [0,T]$, the measure $\mu_{t}:=\mathscr{L}^{d}\res Q$ is a stationary solution of \eqref{eq:PDE}. Therefore, by the superposition principle \cite[Theorem 3.4]{ambrosio2014continuity}, there must exist a family of probability measures $\{\zeta_{x}\}_{x\in \R^{d}}\subset \mathscr{P}(C([0,T];\R^{d}))$ such that $\zeta_{x}$ is concentrated on the set of solutions of \eqref{eq:ODE} with starting point $x$ for almost every $x\in \R^{d}$, and 
    \begin{equation*}
        \mathscr{L}^{d}\res Q=(e_{t})_{\#}\zeta_{x}\otimes d\mathscr{L}^{d}\res Q (x)\qquad \forall t\in [0,T],
    \end{equation*}
    where $e_{t}$ is the evaluation map at time $t$.  Suppose by contradiction that there exists a set $E\subset Q$ with $\mathscr{L}^{d}(E)>0$ such that trajectories of $v$ are unique for every starting point in $E$. Then, for almost every $x\in E$ we would have $\zeta_{x}=\delta_{Y(\cdot, x)}$, therefore,
    \begin{equation*}
        0=\mathscr{L}^{d}(\mathscr{C})\ge \mathscr{L}^{d}(S^{-1}(E))= (e_{T})_{\#}\zeta_{x}\otimes d\mathscr{L}^{d}\res Q (S^{-1}(E))\ge \mathscr{L}^{d}(E)>0,
    \end{equation*}
    a contradiction. 
\end{proof}
The rest of this section is devoted to the proof of \Cref{prop: nonunique traj rev}, which will be carried out in \Cref{subsec:proof-trajectories}, after presenting an overview in \Cref{subsec: overview const traj}.
\subsection{Overview of the construction} \label{subsec: overview const traj}

The construction of the vector field $b$ of \Cref{prop: nonunique traj rev} uses \cite{K2023} as a foundation, which established non-uniqueness of flow maps for divergence-free Sobolev vector fields. 
The idea is to use rescaled copies of the building blocks from \Cref{lem-Building-block} to translate the Cantor cube $p_{n,\sigma}^{\eta}+\ell_{n}^{\eta}Q$ to the final position $s_{n,\sigma}+\ell_{n}^{\eta}Q$, for every $n\ge 1$ and $\sigma \in \mathfrak{S}^{n}$ (recall the notations $\ell_{n}^{\eta},p_{n,\sigma}^{\eta}$ and $s_{n,\sigma}$ introduced in \Cref{subsec:Cantor}).

In \cite{K2023}, the velocity field acts on Cantor cubes of different generations in a serial manner: at any given time, only a single generation of cubes is translated using the building block vector field. While this construction is sufficient for the purposes in \cite{K2023}, a purely sequential approach is inadequate for the objectives of the present work as the following heuristic explanation illustrates. Suppose we wish to use a serial approach for the modulus of continuity $\omega(r)=r|\log(r)|^{1+\eps}$ for some $\eps>0$. Say the $n$th generation is moved in the time interval $\Delta_{n}\subset [0,\infty)$, where $\{\Delta_{n}\}_{n\ge 1}$ have pairwise disjoint interiors.  Call $r_{n}(t)$ the distance of the centers of two adjacent Cantor cubes of the $n$th generation at time $t$. $r_{n}$ must increase from $\ell_{n-1}^{\eta}/2=2^{-2n-1}$ to $2^{-n}$ in time $\mathscr{L}^{1}(\Delta_{n})$. In order to maintain the semi-norm $[b(t,\cdot)]_{C^{\omega}(\R^{d})}$ uniformly bounded in time, we need to impose $\dot{r}_{n}(t)\lesssim \omega(r_{n}(t))=r_{n}(t)|\log(r_{n}(t))|^{1+\eps}$. Integrating this differential inequality in the time interval $\Delta_{n}$ we get
\begin{equation*}
    \mathscr{L}^{1}(\Delta_{n})\gtrsim \int_{2^{-2n-1}}^{2^{-n}}\frac{1}{r|\log(r)|^{1+\eps}}dr\approx \frac{1}{n^{\eps}}.
\end{equation*}
Now, the total time needed for a serial construction is the sum of all $\mathscr{L}^{1}(\Delta_{n})$, which diverges for $\eps\le 1$. Consequently, the best modulus of continuity achievable by such a construction is $\omega(r) = r |\log r|^2$. This coincides with the result of \cite{elgindi2024norm}, which likewise based on a serial construction. 

\begin{figure}[H]
\centering
 \includegraphics[scale = 0.5]{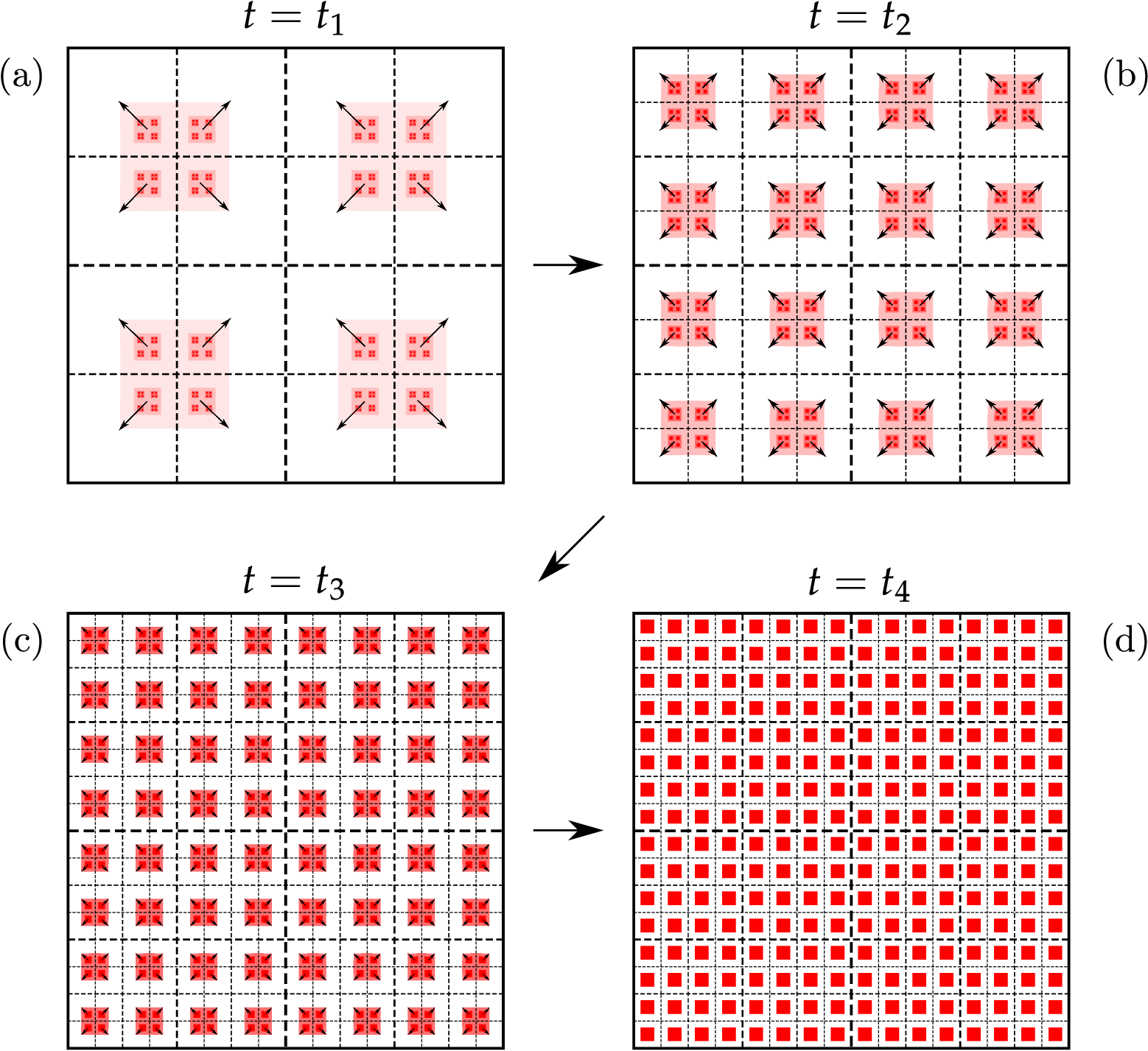}
 \caption{ illustrates the positions and sizes of Cantor cubes of different generations as they are translated in space by the velocity field $b$. The picture resembles Figure 2 in \cite{K2023}. However, there is one crucial difference. In our construction, Cantor cubes from all generations move simultaneously, rather than sequentially. As a result, the size of a Cantor cube at a given generation must expand in time to accommodate the increasing separation of the smaller cubes nested within it. This is clearly visible in the figure. Throughout the construction, the ratio of sizes between successive generations is maintained at $1/4$. When the $n$th generation cubes complete their translation at time $t = t_n$, their size becomes one half of the size of the dyadic cube of the same generation.}
 \label{fig: time evolution ode}
\end{figure}

To surpass this restriction, in this paper we introduce the idea of parallelization, illustrated in \Cref{fig: time evolution ode} and \Cref{fig: parallel}. Unlike the serial approach, the velocity field designed here acts simultaneously on multiple generations of Cantor cubes. More precisely, while the cubes of the $n$th generation are being spread to their designated positions, the cubes of the $(n+1)$th generation, nested within the $n$th generation cubes, are translated at the same time. This idea continues recursively: as the 
$(n+1)$th generation cubes are spread, so do the $(n+2)$th generation cubes contained within them and so on. This multiscale structure of the vector field is crucial for proving Theorem \ref{thm:main} and does not appear in earlier explicit constructions concerning mixing, anomalous dissipation or examples of non-uniqueness in the transport or scalar diffusion equations.

 \begin{figure}[h]
\centering
 \includegraphics[scale = 0.35]{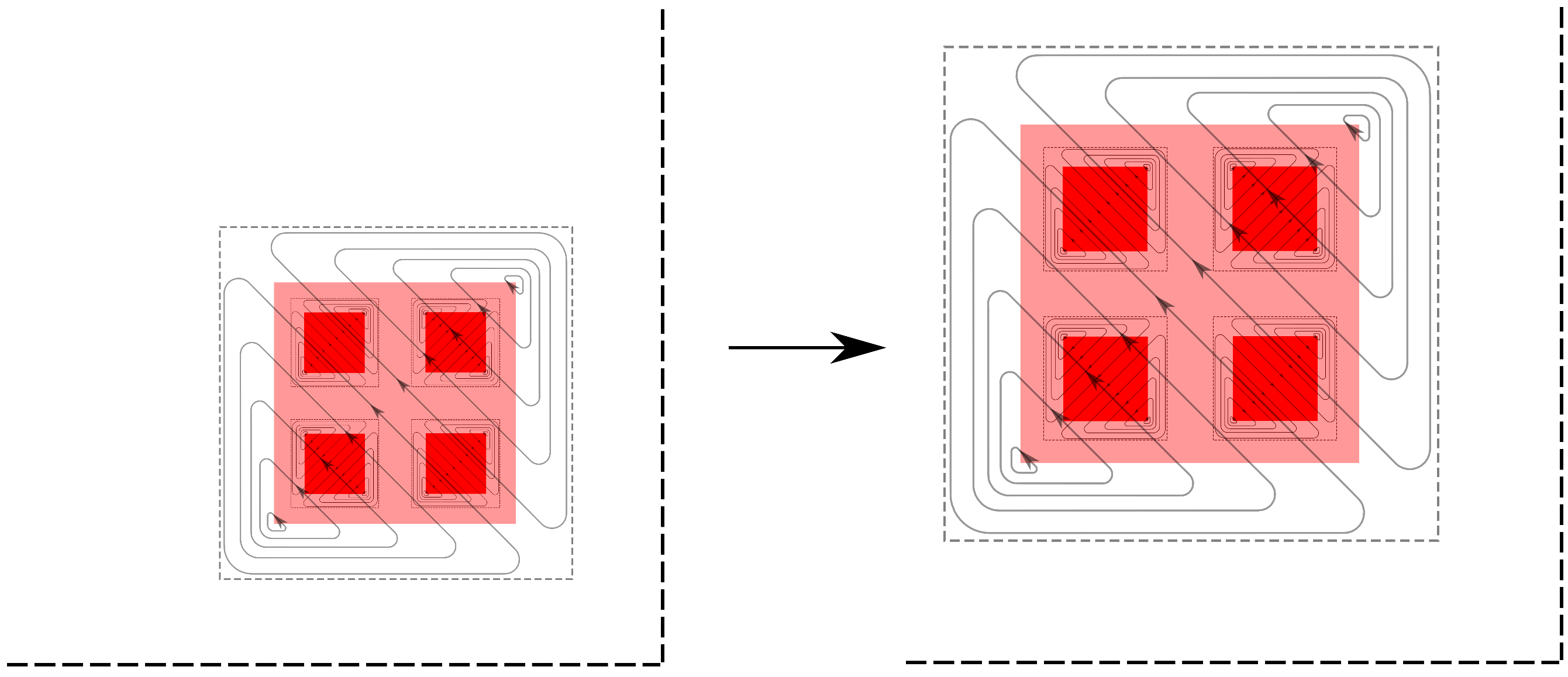}
 \caption{ depicts the superposition of building block velocity fields that translate Cantor cubes of two successive generations. For clarity, only two levels of superposition are shown. In the actual construction, this superposition extends across infinitely many generations. The figure emphasize two essential properties of the construction. (i) Even under superposition, each inner Cantor cube is subject only to a uniform velocity field and therefore translates rigidly. (ii) As the separation between inner cubes grows during the evolution, the size of the enclosing cube must increase to maintain the nested structure of the Cantor set. Consequently, the outer cube expands as it translates.
 }
 \label{fig: parallel}
\end{figure}

We work with a sequence of increasing times $\{t_1, t_2, \dots \}$ (see \eqref{eq: time series} for the exact definition) converging to a final time $T$. The Cantor cubes of the $n$th generation are translated over time $[0, t_n]$. Despite the fact that the vector field we construct is multiscale, all Cantor cubes in the construction undergo rigid translations, just as in \cite{K2023}. This is because a Cantor cube of the $n$th generation only sees the uniform portion of the building blocks carrying the Cantor cubes of generation $n$ and lower. Of course, the uniform component is zero if a lower generation has finished translating. One important aspect of our construction is that, as the 
$(n+1)$th generation cubes spread out, the $n$th generation cubes must simultaneously expand in size to maintain the nested structure of the Cantor cubes. Indeed, In the construction provided in the proof of \Cref{prop: nonunique traj rev}, we ensure that, for all times $t \in [0, t_n]$, the size of $(n+1)$th generation cubes is exactly $1/4$ of that of the $n$th generation cubes.

Although our construction parallelizes the motion of infinitely many generations of Cantor cubes, this choice is made primarily due to mathematical convenience of writing the proof. We could very well have parallelized the motion of cubes in finite batches. For instance, the first batch could involve parallel motion of cubes with indices from $1 = N^k_1$ to $N^k_2$, the second batch from  $N^k_2 + 1$ to $N^k_3$ and so on, for a suitably fast-growing sequence $\{N^k_1, N^k_2, \cdots\}$. In fact, this viewpoint will be used in the construction presented in the next section.

\subsection{Proof of \Cref{prop: nonunique traj rev}}\label{subsec:proof-trajectories}

First of all, up to a suitable time reparametrization, we may reduce to prove the result for some $T>0$. 
We divide the proof into three steps. In the first step, we outline the construction of the vector field $b$. In the second step, we verify that $b$ satisfies the required properties, namely, support in $Q$, divergence-free nature and the prescribed modulus of continuity. Finally, in the third step, we show that the associated flow map $X$ has the desired mapping property.

    \smallskip
    \noindent \textbf{Step 1: Construction of $\bm b$.} First of all we consider an auxiliary non-Osgood modulus of continuity $\tilde{\omega}\le \omega$ such that $\omega(r)/\tilde{\omega}(r)\to \infty$ as $r\to 0^{+}$, as given by \Cref{lem:auxiliary-moc}. We introduce the function $\tilde\Omega:[0,\infty)\to [0,\infty)$ defined as
    \begin{equation*}
        \tilde\Omega(r):= \int_{0}^{r}\frac{1}{\tilde\omega(s)}ds\qquad \forall r\in [0,\infty).
    \end{equation*}
    Note that $\tilde\Omega$ is well defined thanks to the fact that $\tilde\omega$ does not satisfy the Osgood condition \eqref{eq:Osgood-condition}. Moreover, $\tilde\Omega\in C^{1}_{\loc}((0,\infty))$ is strictly increasing and $\tilde\Omega(0)=0$. We define the strictly increasing sequence of times 
    \begin{equation} \label{eq: time series}
        t_{1}:=0,\qquad t_{n+1}:=t_{n}+ \tilde\Omega(2^{-n-1})-\tilde\Omega(2^{-n-2})\quad \forall n\ge 1,\qquad T:= \overunderset{\infty}{n=1}{\sum}(t_{n+1}-t_{n})= \tilde\Omega(1/4)<+\infty.
    \end{equation}
As mentioned in the overview in \Cref{subsec: overview const traj}, our construction is such that
\begin{enumerate}[label=(\alph*)]
\item  on the time interval $[0, t_n]$, the Cantor cubes of the $n$th generation translate and reach their final destination where their centers align with the centers of same generation dyadic cubes. 
\item it preserves the nested structure of Cantor cubes while translating, i.e., if two Cantor cubes of order $n<n'$ are such that the second is included in the first at $t = 0$, then the same containment holds for all times $t\in [0,t_{n}]$.
\end{enumerate}

We now prescribe how the distance between two adjacent Cantor cubes at the $n$th generation, denoted by $r_n(t)$, evolves in time. In our construction, the side length of each $n$th generation Cantor cube is half the distance between two adjacent cubes of the same generation, i.e., $r_n(t)/2$. We will adopt this convention throughout without further explicit mention. Let $\chi:\mathbb{R} \to [0, 1]$ be a non-decreasing $C^\infty$ function such that $\chi \equiv 0$ for $x \leq 0$, $\chi \equiv 1$ for $x \geq 1$, and $|\dot \chi| \leq 2$. We define
\begin{align}
\chi_{n+1}(t) = (t_{n+1} - t_n) \chi \left(\frac{t - t_n}{t_{n+1} - t_n}\right)\qquad \forall n\ge 1.
\end{align}
It is clear that $\chi_{n+1}(t_n) = 0$, $\chi_{n+1}(t_{n+1}) = t_{n+1} - t_n$, and $|\dot \chi_{n+1}|\le 2\mathbbm{1}_{(t_{n},t_{n+1})}$. 
We  recursively define
\begin{equation}\label{def: rn traj}
    \begin{gathered}
    r_{1}(t):= 2^{-1}\qquad  \forall t\in [0,T], \\[5pt]
       r_{n+1}(t):= \begin{cases}
            2^{-n-1}\quad &\forall t\in [t_{n+1},T],\\
            \tilde\Omega^{-1}\left(\tilde\Omega(2^{-n-2})+\chi_{n+1}(t)\right)\quad &\forall t\in [t_{n},t_{n+1}),\\
            r_{n}(t)/4\quad &\forall t\in [0,t_{n})
        \end{cases}\qquad \forall n\ge 1.
\end{gathered}    
\end{equation}
Instead of $\chi_{n+1}$, we could have made a linear choice, i.e., $\chi_{n+1}(t)$ replaced with $t - t_n$ in the definition \eqref{def: rn traj}. This choice would lead to a velocity field $v \in L^\infty_t C^\omega_x$. However, with the introduction of the smooth function $\chi_{n+1}$, one can readily check that $[0,T]\ni t\mapsto r_{n}(t)$ is a $C^{1}$ function for all $n\ge 1$, which will later allow us to make the velocity field continuous in time. With the definition \eqref{def: rn traj}, the following formulas hold:
    \begin{equation}\label{eq:comparison-rk-rn+1-trajectories}
        r_{k}(t)= 4^{n+1-k}r_{n+1}(t),\quad \dot r_{k}(t)= 4^{n+1-k} \dot\chi_{n+1}(t)\tilde\omega(r_{n+1}(t))\qquad \forall t\in (t_{n},t_{n+1}),\quad \forall k\ge n+1\ge 2.
    \end{equation}

    Next, for every $n\ge 1$ and every $\sigma \in \mathfrak{S}^{n}$, we denote by $c_{n,\sigma}(t)$ the position at time $t$ of the center of the $n$th Cantor generation cube corresponding to the symbol $\sigma$, and we prescribe their motion:
    \begin{equation}\label{eq: centers c}
        c_{n,\sigma}(t):= \sum_{k=1}^{n}\frac{r_{k}(t)}{2}\sigma_{k}\qquad \forall t\in [0,T],\quad \forall n\ge 1,\quad \forall \sigma \in \mathfrak{S}^{n}.
    \end{equation}
  
    In  \Cref{subsec:Cantor}, we introduced $p^{\eta}_{n,\sigma}, s_{n,\sigma}$ to denote centers of initial Cantor cubes and dyadic cubes respectively. Note that
    \begin{equation}\label{eq:endpoints-centers-trajectories}
        c_{n,\sigma}(0)=p^{\eta}_{n,\sigma},\qquad c_{n,\sigma}(t)= s_{n,\sigma}=S^{\eta}(p^{\eta}_{n,\sigma})\quad \forall t\in [t_{n},T],\qquad \forall n\ge 1,\quad \forall \sigma\in \mathfrak{S}^{n}.
    \end{equation}
    Using the fact that $r_{n+1}(t)\le r_{n}(t)/2$ for every $t\in [0,T]$ and $n\ge 1$, one can check that the following holds:
    \begin{equation}\label{eq:distance-centers-different-generations}
        c_{n,\sigma}(t)\not\in c_{k,\tilde{\sigma}}(t)+\frac{3}{4}r_{k}(t)Q\qquad \forall t\in [0,T],\quad \forall k>n\ge 1,\quad \forall \sigma\in \mathfrak{S}^{n},\, \tilde{\sigma}\in \mathfrak{S}^{k}.
    \end{equation}
    The construction also ensures that, at any given time, Cantor cubes of the same generation (with double the side length) remain disjoint. That is, we have 
    \begin{equation}\label{eq:disjoint-supports-same-generation}
        \inter  \left(c_{n,\sigma}(t)+r_{n}(t)Q\right)\cap \inter\left(c_{n,\tilde\sigma}(t)+r_{n}(t)Q\right)=\emptyset\qquad \forall t\in [0,T],\quad \forall n\ge 1,\quad \forall \sigma, \tilde\sigma \in \mathfrak{S}^{n}, \,\sigma \neq \tilde{\sigma}.
    \end{equation}
    Moreover, up to time $t_n$, the Cantor cubes of the $(n+1)$th generation have side length equal to $1/4$ of the side length of Cantor cubes of the $n$th generation. As such, the Cantor cubes of $n$th generation contain $2^{d}$ cubes of (double the size) of $(n+1)$th generation:
\begin{equation}\label{eq:containment-generations}
        c_{n+1,\sigma}(t)+r_{n+1}(t)Q\subset c_{n,\sigma'}(t)+\frac{r_{n}(t)}{2}Q\qquad \forall t\in [0,t_{n}),\quad \forall n\ge 1,\quad \forall \sigma\in \mathfrak{S}^{n+1}.
    \end{equation}
    Recall here the notation $\sigma':=(\sigma_{1},\dots,\sigma_{n})$ for every $\sigma=(\sigma_{1},\dots,\sigma_{n},\sigma_{n+1})\in\mathfrak{S}^{n+1}$, $n\ge 1$.
    
    Finally, recalling the notation $u^{e}$ for the building blocks introduced in \Cref{lem-Building-block}, we are in the position to define the velocity field $b$:
\begin{align}\label{eq:def-b-trajectories}
    b(t,x):= \overunderset{\infty}{n=1} {\sum} \; \underset{\sigma \in \mathfrak{S}^{n}}{\sum} \frac{\dot{r}_{n}(t)}{2}\,u^{\sigma_{n}}\left(\frac{x-c_{n,\sigma}(t)}{r_{n}(t)}\right)\qquad \forall (t,x)\in [0,T]\times \R^{d}.
\end{align}

\smallskip
\noindent \textbf{Step 2: Properties of the vector field.}  
We first prove that $b_{t}$ is well-defined for all $t\in [0,T)$, and that $b\in C([0,T);C^{\omega}(\R^{d}))$. 
Let us fix $m\ge 1$ and $t\in [t_{m},t_{m+1})$. We show that the series defining $b_t$ in \eqref{eq:def-b-trajectories} is pointwise absolutely convergent, which shows that $b_{t}$ is well-defined. Note that $\dot r_{n}(t)=0$ for all $n\le m$. By \Cref{lem-Building-block}, $u^{\sigma_{n}}((\cdot-c_{n,\sigma}(t))/r_{n}(t))$ appearing in \eqref{eq:def-b-trajectories} is zero outside $\inter\left(c_{n,\sigma}(t)+r_{n}(t)Q\right)$. Then, thanks to \eqref{eq:disjoint-supports-same-generation}, for every $x\in \R^{d}$ and $n\ge m+1$ there is at most one $\sigma\in \mathfrak{S}^{n}$ for which $u^{\sigma_{n}}((x-c_{n,\sigma}(t))/r_{n}(t))\neq 0$. Therefore, the sum defining $b(t,x)$ in \eqref{eq:def-b-trajectories} is absolutely convergent and we have
\begin{equation}\label{eq:uniform-bound-bt}
    \lVert b_{t}\rVert_{L^{\infty}(\R^{d})}\lesssim_{d} \sum_{n\ge m+1}\dot r_{n}(t)=\sum_{n\ge m+1}4^{m+1-n} \dot\chi_{m+1}(t) \tilde\omega(r_{m+1}(t))\lesssim_{d}\tilde\omega(r_{m+1}(t))\qquad \forall t\in [t_{m},t_{m+1}),
\end{equation}
where we also used \eqref{eq:comparison-rk-rn+1-trajectories}. Next we show that the series of space derivatives is absolutely convergent as well. Notice that $\nabla u^{\sigma_{n}}((\cdot-c_{n,\sigma}(t))/r_{n}(t))$ is zero outside $\inter\left(c_{n,\sigma}(t)+r_{n}(t)Q\right)\setminus (c_{n,\sigma}(t)+(r_{n}(t)/2)Q)$. This, together with \eqref{eq:disjoint-supports-same-generation} and \eqref{eq:containment-generations} shows that for all $x\in \R^{d}$, there is at most one choice of $n\ge m+1$ and $\sigma \in \mathfrak{S}^{n}$ for which $\nabla u^{\sigma_{n}}((x-c_{n,\sigma}(t))/r_{n}(t))\neq 0$. Therefore, using \eqref{eq:comparison-rk-rn+1-trajectories} we find
\begin{equation}\label{eq:Lip-bound-bt}
    \lVert \nabla b_{t}\rVert_{L^{\infty}(\R^{d})}\lesssim_{d} \sup_{n\ge m+1} \frac{\dot r_{n}(t)}{r_{n}(t)}\lesssim  \frac{\tilde\omega(r_{m+1}(t))}{r_{m+1}(t)}\qquad \forall t\in [t_{m},t_{m+1}).
\end{equation}
By \Cref{lem:interpolation}, \eqref{eq:uniform-bound-bt} and \eqref{eq:Lip-bound-bt} together give that $[0,T)\ni t\mapsto b_{t}$ is uniformly bounded in $C^{\tilde{\omega}}(\R^{d})$, and so in $C^{\omega}(\R^{d})$, as $\tilde{\omega}\le \omega$. The continuity in time is because of continuity of $r_n(t)$, $c_{n, \sigma}(t)$ and $\dot{r}_n(t)$ combined with the regularity of the building blocks $u^{e}$.

We now show continuity up to time $t=T$. Note that $b_{T}\equiv 0$ because $\dot{r}_{n}(T)=0$ for all $n\ge 1$. Let us prove that $\lVert b_{t}\rVert_{C^{\omega}(\R^{d})}\to 0$ as $t\to T^{-}$. By \eqref{eq:uniform-bound-bt} and $\tilde{\omega}(r_{m+1})\le \tilde{\omega}(2^{-m-1})\to 0$ as $m\to \infty$ we get $\lVert b_{t}\rVert_{L^{\infty}(\R^{d})}\to 0$ as $t\to T^{-}$. To prove $[b_{t}]_{C^{\omega}(\R^{d})}\to 0$ as $t\to T^{-}$ we multiply and divide by $\omega(r_{n+1}(t))$ in both \eqref{eq:uniform-bound-bt} and \eqref{eq:Lip-bound-bt} and then use \Cref{lem:interpolation} to find 
\begin{equation*}
    \sup_{t\in [t_{n},t_{n+1})}[ b_{t}]_{C^{\omega}(\R^{d})}\lesssim_{d} \sup_{t\in [t_{n},t_{n+1})}\frac{\tilde\omega(r_{n+1}(t))}{\omega(r_{n+1}(t))}\to 0\qquad \text{as $n\to \infty$}
\end{equation*}
thanks to the choice of $\tilde{\omega}$, which is infinitesimal with respect to $\omega$ at zero.  

Finally, $b_{t}\equiv 0$ in $\R^{d}\setminus Q$ for all $t\in [0,T]$ because all the terms appearing in the sum \eqref{eq:def-b-trajectories} are supported in $Q$. Similarly, $\diver b_{t}\equiv 0$ for all $t\in [0,T]$ because of the incompressibility of the building blocks $u^{e}$ from \Cref{lem-Building-block}.

\smallskip
\noindent \textbf{Step 3: Properties of the flow map.} By standard Cauchy-Lipschitz theory, the bounds in \eqref{eq:uniform-bound-bt} and \eqref{eq:Lip-bound-bt} ensure the existence and uniqueness of a continuous flow map $X:[0,T)\times \R^{d}\to \R^{d}$. Furthermore, since $b$ is uniformly bounded in $[0,T]\times \R^{d}$, $X$ is uniquely extended continuously to the whole $[0,T]\times \R^{d}$.  

We prove that $X(T, x) = S^{\eta}(x)$ for all $x \in \mathscr{C}^{\eta}$. By the continuity of the flow map, it is sufficient to show that
\begin{equation}\label{eq:goal-mapping-centers}
    X(T, p^{\eta}_{n,\sigma}) = s_{n,\sigma}\qquad \forall n\ge 1,\quad \forall \sigma\in \mathfrak{S}^{n},
\end{equation}
which basically means the center for Cantor cubes of different generation are mapped to the center of dyadic cubes of the same generation. This is because any $x \in \mathscr{C}^{\eta}$ can be arbitrarily well-approximated by centers of Cantor cubes. Note that by \eqref{eq:endpoints-centers-trajectories} and the uniqueness of trajectories, \eqref{eq:goal-mapping-centers} is a consequence of 
\begin{equation}\label{eq:trajectories-centers}
    \dot c_{n,\sigma}(t)=b(t,c_{n,\sigma}(t))\qquad \forall t\in (0,T),\quad \forall n\ge 1,\quad \forall \sigma\in \mathfrak{S}^{n}.
\end{equation}

In the rest of the proof we show \eqref{eq:trajectories-centers}. Let us fix $n\ge 1$ and $\sigma\in \mathfrak{S}^{n}$, and consider $t\in [t_{m-1},t_{m})$ for some $m\ge 2$.  
Note that $\dot r_{k}(t)=0$ for all $k< m$. Moreover, by \eqref{eq:distance-centers-different-generations} and \Cref{lem-Building-block} we get $u^{\tilde\sigma_{k}}((c_{n,\sigma}(t)-c_{k,\tilde{\sigma}}(t))/r_{k,\tilde{\sigma}}(t))=0$ for all $k>n$ and all $\tilde{\sigma}\in \mathfrak{S}^{k}$. 
This two considerations together give
\begin{equation*}
    b(t,c_{n,\sigma}(t))=\sum_{k=m}^{n}\sum_{\tilde{\sigma}\in \mathfrak{S}^{k}}\frac{\dot{r}_{k}(t)}{2}\,u^{\tilde\sigma_{k}}\left(\frac{c_{n,\sigma}(t)-c_{k,\tilde\sigma}(t)}{r_{k}(t)}\right).
\end{equation*}
In particular, if $m\ge n+1$ we get $b(t,c_{n,\sigma}(t))=0=\dot c_{n,\sigma}(t)$. Suppose now that $m\le n$. Let us call $\sigma^{k}\in \mathfrak{S}^{k}$ the unique symbol such that $\sigma^{k}\subset \sigma$, for all $k\le n$. From \eqref{eq:containment-generations} we deduce 
\begin{align}
c_{n, \sigma}(t) \in c_{n, \sigma}(t) + \frac{r_n(t)}{2} Q \subset c_{n-1, \sigma^{n-1}}(t)+ \frac{r_{n-1}(t)}{2} Q\subset\dots \subset  c_{m, \sigma^m}(t)+  \frac{r_m(t)}{2} Q. 
\end{align}
In particular, by \eqref{eq:disjoint-supports-same-generation}, $c_{n,\sigma}(t)\not\in c_{k,\tilde{\sigma}}(t)+\frac{3}{4}r_{k}(t)Q$ for all $k\in \{m,\dots, n\}$ and all $\tilde{\sigma}\in \mathfrak{S}^{k}$, $\tilde{\sigma}\neq \sigma^{k}$. Therefore, from \Cref{lem-Building-block} we obtain
\begin{align} \label{eq: b part 1}
b(t, c_{n, \sigma}(t)) = \overunderset{n}{k=m}\sum \frac{\dot r_{k}(t)}{2}u^{\sigma^{k}_{k}}\left(\frac{c_{n,\sigma}(t)-c_{k,\sigma^{k}}(t)}{r_{k}(t)}\right)= \overunderset{n}{k=m}\sum \frac{\dot r_{k}(t)}{2}\sigma_{k}=\dot c_{n,\sigma}(t),
\end{align}
as desired. This concludes the proof of \Cref{prop: nonunique traj rev}.

\section{Non-uniqueness of transported densities} \label{sec: proof main thm}
The goal of this section is to present a construction based on a fixed-point iteration combined with the ideas developed in the previous section, and prove the following proposition. The main theorem (Theorem \ref{thm:main}) is then a consequence of this proposition and Ambrosio's superposition principle.
\begin{prop}\label{prop:non-uniqeness-L1}
    Let $d\ge 2$ and $\omega:[0,\infty)\to [0,\infty)$ be any non-Osgood modulus of continuity. Given any $T > 0$, there exists a divergence-free velocity field $v \in C([0, T]; C^\omega(\mathbb{R}^d))$, such that $\supp v(t, \cdot) \subseteq Q$ for every $t\in [0,T]$,
    for which there exists a non-negative solution $\mu_{t}$ of \eqref{eq:PDE} with $\supp \mu_{t}\subseteq Q$ for all $t\in [0,T]$, $\mu_{t}\ll \mathscr{L}^{d}$ for almost every $t\in [0,T]$, and such that $\mu_{0}\equiv \mathbbm{1}_{Q}\mathscr{L}^{d}$ and $\mu_{T}\not\equiv \mathbbm{1}_{Q}\mathscr{L}^{d}$.
\end{prop}

\begin{proof}[Proof of \Cref{thm:main}]
    Let $v$ and $\mu$ be the outcome velocity field and solution from \Cref{prop:non-uniqeness-L1}. We set $$\mu^{1}_{t}:=\mu_{t},\qquad \mu^{2}_{t}:= \mathbbm{1}_{Q}\mathscr{L}^{d}\qquad \forall t\in [0,T].$$
    Both $\mu^{1},\mu^{2}\in C_{w^*}([0,T];\mathcal{M}_{+}(\R^{d}))$ are solutions of \eqref{eq:PDE} with divergence-free velocity field $v\in C([0, T]; C^\omega(\mathbb{R}^d))$ and initial datum $\mathbbm{1}_{Q}\mathscr{L}^{d}$. Moreover, they are distinct, as $\mu^{1}_{T}=\mu_{T}\neq \mathbbm{1}_{Q}\mathscr{L}^{d}=\mu^{2}_{T}$. They are supported in $Q$ and are absolutely continuous with respect to Lebesgue for almost every time. Finally, if $v$ had a unique trajectory for almost every starting point in $Q$, then Ambrosio's superposition principle \cite[Theorem 3.4]{ambrosio2014continuity} would imply that $\mu^{1}\equiv \mu^{2}$, as $\mu^{1}_{0}=\mu^{2}_{0}=\mathbbm{1}_{Q}\mathscr{L}^{d}$, which is not the case. This shows that $v$ has non-unique trajectories in a set of positive measure and concludes the proof.
\end{proof}

The rest of this section is devoted to the proof of \Cref{prop:non-uniqeness-L1}. After giving an overview of the construction in \Cref{sec: overview L1}, we introduce several auxiliary notations in \Cref{subsec:preliminary-notations}. Then, \Cref{subsec:velocity-field-construction} and \Cref{subsec: density construction} are devoted to the construction and the proof of the required properties of the velocity field $v$ and the solution $\mu$, respectively. The proof of \Cref{prop:non-uniqeness-L1} is then concluded in \Cref{subsec:proof-nonuniqueness-L1}  

\subsection{Overview of the construction} \label{sec: overview L1}
We consider a strictly increasing sequence of numbers $\{\eta_k\}_{k \ge 1} \subset [1,\infty)$ and an associated sequence of lengths $\{\ell_k\}_{k \ge 0}$ such that
\begin{equation*}
    \ell_0 = 1,\qquad \ell_{k+1} = \frac{\ell_k}{2^{1 +\eta_{k+1}}}\le \frac{\ell_{k}}{4}\quad \forall k\ge 1.
\end{equation*}
The core of the proof of \Cref{prop:non-uniqeness-L1} relies on constructing a family of divergence-free velocity fields $\{\bar B^k\}_{k \ge 0}$ together with corresponding solutions of the \eqref{eq:PDE} $\{\bar \Theta^k\}_{k \ge 0}$. For each $k$, the solution $\bar \Theta^k$ is supported at time $t=0$ on $2^d$ cubes of side length $\ell_{k+1}$, and by time $t=1$ it becomes uniformly distributed inside a cube of length $\ell_k$ (see \Cref{fig: time series 1}). Each velocity field $\bar B^k$ belongs to $C([0,1]; C^\omega(\mathbb{R}^d))$ and the associated solution $\bar \Theta^k$ is absolutely continuous with respect to Lebesgue measure for almost every time. The desired velocity field $v$ and the corresponding solution $\mu$ in \Cref{prop:non-uniqeness-L1} are then obtained from $\bar B^{0}$ and $\bar \Theta^{0}$, respectively, after reversing and reparametrizing time.  

 \begin{figure}[H] 
\centering
\includegraphics[scale = 0.35]{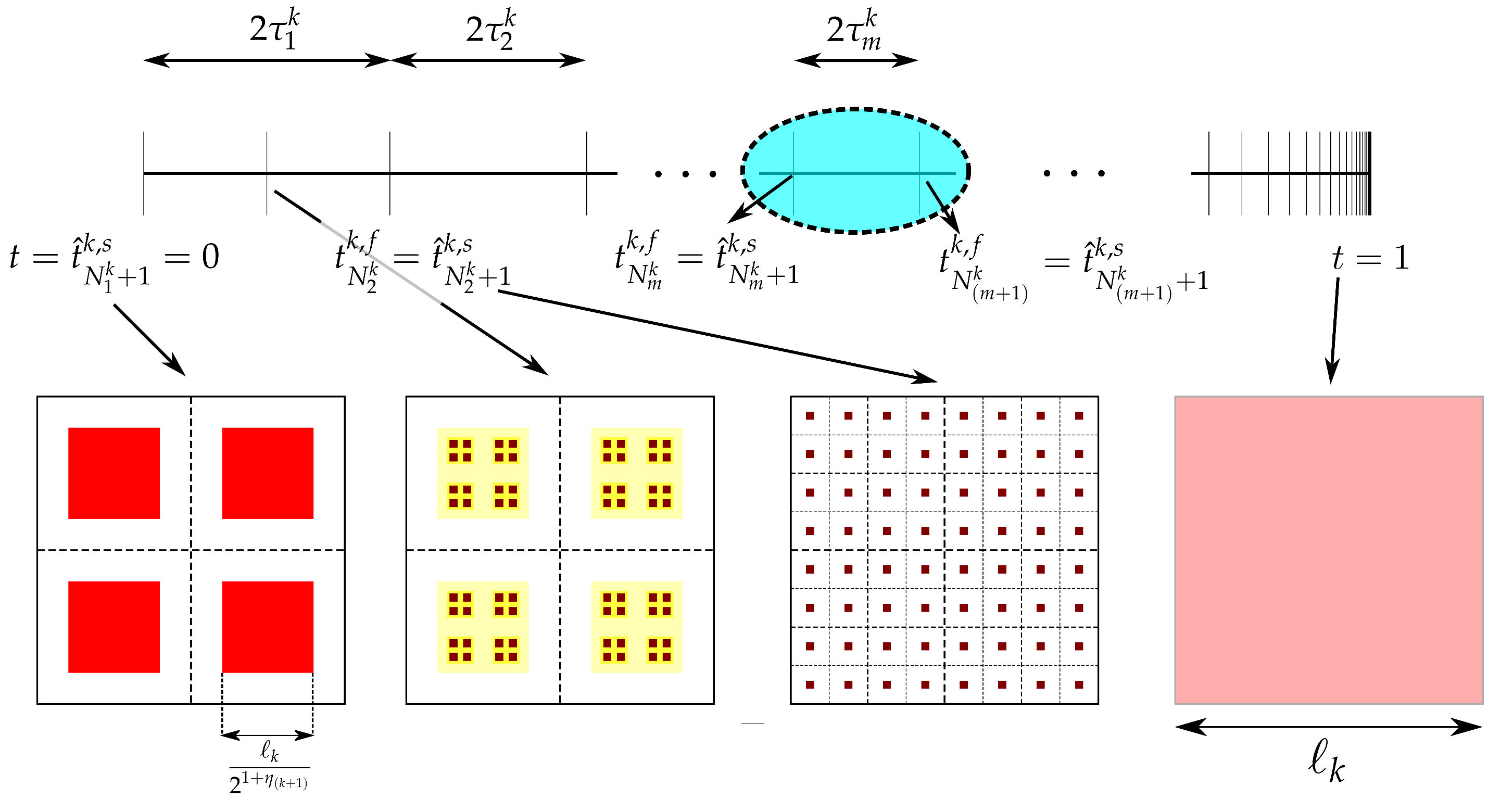}
 \caption{ Time decomposition of the interval $[0, 1]$ into subintervals of length $2 \tau^k_m$ for $m \in \mathbb{N}$. At time $t = 0$, the density is concentrated on cubes of side length $\ell_{k+1}$. By the time $t = 1$, the density is uniformly distributed inside a cube of side length $\ell_k$. In the first half of the interval of length $2 \tau^k_1$, the  ``breaking'' process takes place. At time $t = \hat{t}^{k, f}_{N^k_2}$, the 
 density concentrates on $2^{N^k_2 d}$ cubes. These cubes are not yet uniformly distributed. Instead, they are contained in several generations of nested fictitious cubes (shown in yellow in Panel 2), in an arrangement analogous to the Cantor-type construction from the previous section. In the second half of the interval, by time $t=t^{k,f}_{N_2^k}$, the small cubes are translated in parallel and arranged uniformly on a dyadic grid (as illustrated in Panel 3). The same ``breaking'' and ``parallel translation'' procedure is repeated on each interval of length $2\tau_m^k$. \Cref{fig: time series 2} illustrates the corresponding subdivision of a generic $2\tau_m^k$ interval.
 }
 \label{fig: time series 1}
\end{figure}

The construction of the velocity fields $\{\bar B^k\}_{k\ge 0}$ and the associated solutions $\{\bar\Theta^k\}_{k\ge 0}$ is based on a fixed-point iteration in suitable complete metric spaces (see \eqref{def: space B} and \eqref{def: space tau}), following an approach that is similar in spirit to \cite{BCK2024}. The necessity of working with a family, rather than a single velocity field and solution, arises from the requirement of closing the fixed-point iteration (see more on this below).

For every $k \geq 0$, we work on the time interval $[0,1]$, which is further subdivided into intervals of length $2\tau^k_m$ for $m \ge 1$, with $\sum_{m\ge 1} 2 \tau^k_m = 1$. We also introduce a strictly increasing sequence of indices $\{N^k_m\}_{m \geq 1}\subset \N$ with $N^k_1 = 1$. The parameters $\eta_k$, the length $\ell_k$ and the sequence $\{N^k_m\}_{m \geq 1}$ will be chosen recursively in $k$ (see Section \ref{sec: L1 time series}). The role of the sequence $\{N^k_m\}_{m \geq 1}$ is the following. On the first half of each interval of length $2\tau^k_m$ (the interval $[\hat{t}^{k,s}_{N^k_m + 1}, \hat{t}^{k,f}_{N^k_{m+1}}]$), we ``break cubes'' exactly $(N^k_{m+1} - N^k_m)$ times. On the second half (the interval $[t^{k,s}_{N^k_{m+1}}, t^{k,f}_{N^k_{m+1}}]$) of the same interval, we ``parallelly translate'' the fictitious cubes with indices ranging from $N^k_m + 1$ to $N^k_{m+1}$. This is a major difference from the $L^1$ construction in \cite{BCK2024}. There, once a given generation of cubes is ``broken'' into the next generation, the resulting cubes are immediately translated so as to be uniformly distributed on a dyadic grid. In contrast, in our construction the breaking process takes place over several generations and then followed by a parallelized translation, similar to Section \ref{sec:proof weak}.

The interval $[\hat{t}^{k,s}_{N^k_m + 1}, \hat{t}^{k,f}_{N^k_{m+1}}]$  is further divided equally into $(N^k_{m+1} - N^k_m)$ subintervals $[\hat{t}^{k,s}_n, \hat{t}^{k,f}_n]$ as shown in \Cref{fig: time series 2}, where $n \in \{N^k_m + 1, \dots, N^k_{m+1}\}$. Each such interval therefore has length $\tau^k_m / (N^k_{m+1} - N^k_m)$. At the beginning of the interval $[\hat{t}^{k,s}_{N^k_m + 1}, \hat{t}^{k,f}_{N^k_{m+1}}]$, at time $t = \hat{t}^{k,s}_{N^k_m + 1}$, the solution $\bar\Theta^k$ consists of $2^{N^k_m d}$ cubes, each of side length $\ell_{k + N^k_m}$. Over the time interval $[\hat{t}^{k,s}_{N^k_m + 1}, \hat{t}^{k,f}_{N^k_m + 1}]$, each of these cubes ``breaks'' into $2^{d}$ smaller cubes, resulting in a total of $2^{(N^k_m + 1)d}$ cubes of side length $\ell_{k + N^k_m + 1}$. Proceeding similarly, over the interval $[\hat{t}^{k,s}_{N^k_m + 2}, \hat{t}^{k,f}_{N^k_m + 2}]$, the new cubes break, resulting in $2^{(N^k_m + 2)d}$ cubes of side length $\ell_{k + N^k_m + 2}$. This process continues iteratively until time $t = t^{k,f}_{N^k_{m+1}}$, at which point $\bar\Theta^k$ is composed of $2^{N^k_{m+1} d}$ cubes, each of side length $\ell_{k + N^k_{m+1}}$.

 \begin{figure}[H]
\centering
 \includegraphics[scale = 0.22]{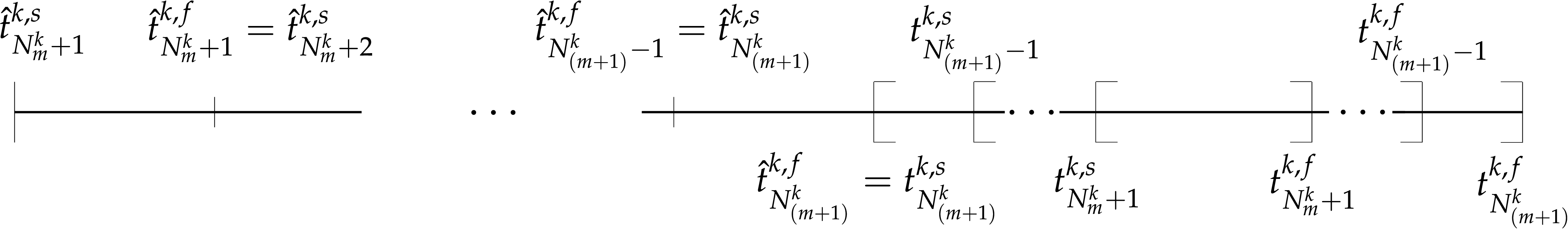}
 \caption{  shows an expanded view of an interval of length $2 \tau^k_m$ from \Cref{fig: time series 1}. It is divided into two halves. The first half consists of consecutive subintervals of the form $[\hat{t}^{k,s}_n, \hat{t}^{k,s}_n)$, while the second half consists of nested subintervals of the same form $[\hat{t}^{k,s}_n, \hat{t}^{k,s}_n)$ for $n \in \{N^k_m +1, \cdots N^k_{m+1}\}$. }
 \label{fig: time series 2}
\end{figure}

At this stage, however, the cubes are no longer distributed uniformly in space. Throughout the breaking process, each cube of side length $\ell_{k + N^k_{m+1}}$ at time $t = t^{k,f}_{N^k_{m+1}}$ stems from a unique sequence of nested cubes with side lengths $\ell_{k + n}$ for $n \in \{N^k_m + 1 , \dots, N^k_{m+1}\}$. Consequently, we regard the cubes of size $\ell_{k + N^k_{m+1}}$ as being embedded in a nested hierarchy of fictitious cubes, where the $n$th generation consists of cubes of side length $\ell_{k + n}$ for $n \in \{N^k_m + 1, \dots, N^k_{m+1}\}$ as illustrated in the second panel of \Cref{fig: time series 1}. This viewpoint is analogous to that of Section \ref{sec:proof weak}, where Cantor points are interpreted as lying within a nested structure of Cantor cubes, the essential difference being that in the present setting only finitely many generations are involved.

Once the cubes of side length $\ell_{k + N^k_{m+1}}$ are viewed as being embedded within multiple generations of fictitious cubes, we then perform a parallelized translation of the fictitious cubes on the interval $ [t^{k,s}_{N^k_{m+1}}, t^{k,f}_{N^k_{m+1}}]$, following the strategy of Section \ref{sec:proof weak}. This translation redistribute the cubes of side length $\ell_{k + N^k_{m+1}}$ uniformly on a dyadic grid.
There are, however, two differences compared to Section \ref{sec:proof weak}. The first is, of course, that here we parallelize the translation only across finitely many generations of fictitious cubes. The second, and more crucial difference, is that in the present construction the ratio $\ell_{k^\prime+1}/\ell_{k^\prime}$ is strictly less than $1/4$ and not constant. This does not pose any mathematical difficulty but only a technical one, which is addressed as follows.
At time $t = t^{k,s}_{N^k_{m+1}}$, the smallest fictitious cubes, of side length $\ell_{k + N^k_{m+1}}$, begin translating and expanding first, until their size reaches $1/4$ of that of the fictitious cubes of generation $n = N^k_{m+1} - 1$. Then, at time $t = t^{k,s}_{N^k_{m+1} - 1}$, both generations $n = N^k_{m+1} - 1$ and $n = N^k_{m+1}$ translate and expand in parallel (while maintaining $1/4$ size ratio), until the size of the $(N^k_{m+1} - 1)$th generation becomes one quarter of that of generation $N^k_{m+1} - 2$. This procedure continues until time $t = t^{k,s}_{N^k_m + 1}$, at which point the ratio of the size of all consecutive generations becomes $1/4$ and the situation becomes analogous to that of Section \ref{sec:proof weak}. From this time onward, all generations $n \in \{N^k_m + 1, \dots, N^k_{m+1}\}$ translate in parallel until $t = t^{k,f}_{N^k_m + 1}$, when the fictitious cubes of generation $N^k_m + 1$ 
(the largest cubes) reach their destination. The remaining generations continue to expand and translate in parallel until, at time $t = t^{k,f}_{N^k_m + 2}$, the fictitious cubes of generation $N^k_m + 2$ also reach their destination. This process continues until $t = t^{k,f}_{N^k_{m+1}}$, when the fictitious cubes of generation $N^k_{m+1}$ reach their destination. At this final time, $\bar\Theta^k$ is composed of $2^{N^k_{m+1}d}$ cubes, each of side length $\ell_{k + N^k_{m+1}}$, distributed uniformly on a dyadic grid. We emphasize that whenever we refer to the expansion of cubes, this applies only to the fictitious cubes. The cubes in which $\bar \Theta^k$ is supported, never change their size. Instead, they are transported rigidly by the translations of the fictitious cubes in which they are embedded.

We conclude the overview by briefly describing the ``breaking'' process. As mentioned earlier, for any $n \geq 2$, at time $t = \hat{t}^{k,s}_n$ the density $\bar \Theta^k$ consists of $2^{(n-1)d}$ cubes of side length $\ell_{k+n-1}$. By time $t = \hat{t}^{k,f}_n$, each of these cubes ``breaks'' into $2^{d}$ smaller cubes, so that $\bar\Theta^k$ is then composed of $2^{nd}$ cubes of side length $\ell_{k+n}$. In the definition of $\bar B^k$, this breaking step is precisely where we use $2^{(n-1)d}$ reverse-in-time and appropriately time-scaled copies of the vector field $\bar B^{k+n-1}$. This naturally motivates a construction based on a fixed-point argument. A crucial point to note is that one cannot simply use spatially scaled copies of a single velocity field unless the ratio $\ell_{k+1} / \ell_k$ remains constant in $k$. This was the case in the $L^1$ construction of \cite{BCK2024} and, as such, a fixed point iteration based on a single velocity field was sufficient. In contrast, in the present construction, depending on the modulus of continuity $\omega$, we are required to choose lengths such that $\ell_{k+1} / \ell_k$ diverges as $k \to \infty$ to close the fixed-point iteration. This then requires us to work with a family of velocity fields rather than a single one.

\subsection{Preliminary Notations}\label{subsec:preliminary-notations} 

The goal of this section is to introduce all the notations and preliminary constructions that will be used to build the velocity field $v$ and the anomalous solution $\mu$ of \Cref{prop:non-uniqeness-L1}.

\subsubsection{Auxiliary modulus of continuity}
 In addition to the non-Osgood modulus of continuity $\omega:[0,\infty)\to [0,\infty)$, we consider an auxiliary non-Osgood modulus of continuity $\tilde\omega:[0,\infty)\to [0,\infty)$ such that $\tilde{\omega}\le \omega$, and $\omega(r)/\tilde{\omega}(r)\to \infty$ as $r\to 0^{+}$. The existence of such $\tilde{\omega}$ is ensured by \Cref{lem:auxiliary-moc}. 

We call $\tilde\Omega:[0,\infty)\to [0,\infty)$ the function
\begin{equation*}
    \tilde\Omega(r):= \int_{0}^{r}\frac{1}{\tilde\omega(s)}ds\qquad \forall r\in [0,\infty),
\end{equation*}
which is well-defined by the non-Osgood condition on $\tilde\omega$. Note that $\tilde\Omega \in C^{1}_{\loc}((0,\infty))$, it is strictly increasing, and $\tilde\Omega(0)=0$.

We also introduce the non-increasing function $W:(0,\infty)\to [1,\infty)$ given by
\begin{equation} \label{def: weight}
    W(r):= \inf_{s\in (0,r]}\frac{\omega(s)}{\tilde{\omega}(s)} \qquad \forall r\in (0,\infty).
\end{equation}
Later, we will use $W$ as a weight in the definition of complete metrics on the ambient spaces \eqref{def: space B} and \eqref{def: space tau}, in which the fixed-point iteration will be set.  
Note that $W(r)\to \infty$ as $r\to 0^{+}$, thanks to the choice of the auxiliary modulus of continuity $\tilde{\omega}$.

\subsubsection{Length scales} Below we introduce a few objects associated to a given sequence of numbers $\{\eta_{k}\}_{k\ge 1}\subset [1,\infty)$. Later, in \Cref{lem:choice-parameters}, we will make a specific choice of this sequence, subjected to a certain set of constraints.

We denote by $\{\nu_{k}\}_{k\ge 0}$ the sequence of partial sums from $\{\eta_{k}\}_{k\ge 1}$:
\begin{equation*}
    \nu_{0}:= 0,\qquad \nu_{k}:= \sum_{j=1}^{k}\eta_{j}\quad \forall k\ge 1.
\end{equation*}
Next, we define a sequence of lengths $\{\ell_{k}\}_{k\ge 0}$ as follows:
\begin{equation*}
    \ell_{0}:= 1,\qquad \ell_{k}:= 2^{-k-\nu_{k}}\qquad \forall k\ge 1.
\end{equation*} 

\subsubsection{Time series} 
\label{sec: L1 time series}
Corresponding to the sequence $\{\eta_{k}\}_{k\ge 1}$ above, for every $k\ge 0$, we consider a strictly increasing sequence of positive integers $\{N_{m}^{k}\}_{m\ge 1}\subset \N$ such that the following holds:
\begin{equation}\label{eq:choice-of-N^k_m}
    N_{1}^{k}=1,\qquad \sum_{m\ge 1}\tilde{\Omega} \left(\ell_{k}2^{-N_{m}^{k}-1}\right) \; \in \; \left(\tilde{\Omega}\left(\ell_{k}/4\right), \, 2 \, \tilde{\Omega}\left(\ell_{k}/4\right)\right)\qquad \forall k\ge 0.
\end{equation}
Such sequence exists because $\tilde{\Omega}$ is increasing, continuous and $\tilde{\Omega}(0)=0$. As mentioned in the overview in \Cref{sec: overview L1}, we will parallelize the motion of the $m$th batch of fictitious cubes, i.e. cubes with index from $N_{m}^{k}+1$ to $N_{m+1}^{k}$.

Next, for every $k\ge 0$, we define a sequence $\{\bar \tau^{k}_{m}\}_{m\ge 1}\subset (0,\infty)$ given by  
\begin{equation*}
\begin{aligned}
    \bar\tau_{m}^{k}&:= \sum_{n=N_{m}^{k}+2}^{N_{m+1}^{k}}\tilde{\Omega}(\ell_{k+n-2}2^{-3})-\tilde{\Omega}(\ell_{k+n-1}2^{-1})\\
    &\quad \, +\tilde{\Omega}(\ell_{k} 2^{-N_{m}^{k}-1})-\tilde{\Omega}(\ell_{k+N_{m}^{k}}2^{-1})\\
    &\quad \, + \sum_{n=N_{m}^{k}+2}^{N_{m+1}^{k}}\tilde{\Omega}(\ell_{k} 2^{-n})-\tilde{\Omega}(\ell_{k} 2^{-n-1}).
\end{aligned}\qquad \forall m \ge 1, \quad \forall k\ge 0.    
\end{equation*}
Summing telescopically some of the terms in the definition of $\bar \tau_{m}^{k}$ and using the monotonicity of  $\tilde{\Omega}$, we find 
\begin{equation}\label{eq:bound-bartau{m}{k}}
    \bar\tau^{k}_{m}\le 3\tilde{\Omega}(\ell_{k}2^{-N_{m}^{k}-1})\qquad \forall k\ge 0,\quad \forall m\ge 1.
\end{equation}
Thanks to \eqref{eq:choice-of-N^k_m} and \eqref{eq:bound-bartau{m}{k}}, the sequence $\bar{\tau}_{m}^{k}$ is summable in $m$ and we can define
\begin{equation}\label{eq:def-Tk}
    T^{k}:= \sum_{m\ge 1}2\bar\tau_{m}^{k}\le 12\tilde{\Omega}(\ell_{k}2^{-2})\qquad \forall k\ge 0.
\end{equation}
Finally, for every $k \geq 0$, we introduce the following sequence of normalized times:
\begin{equation}\label{eq:def-tau{m}{k}}
    \tau_{m}^{k}:=\frac{\bar{\tau}_{m}^{k}}{T^{k}}\qquad \forall m\ge 1,\quad \forall k\ge 0.
\end{equation}
We have now all the ingredients to explain the constraint on the sequence $\{\eta_{k}\}_{k\ge 1}$: 
\begin{lem}\label{lem:choice-parameters} There exists a sequence $\{\eta_{k}\}_{k\ge 1}\subset [1,\infty)$ such that the following condition holds:
\begin{equation}\label{eq:choice-parameters-fixed-point}
    W(\ell_{k+N_{m}^{k}})\ge 4m\frac{N_{m+1}^{k}-N_{m}^{k}}{\tau_{m}^{k}}W(\ell_{k})\qquad \forall m\ge 1,\quad \forall k\ge 0.
\end{equation}
\end{lem}
\begin{proof}
We proceed inductively as follows. Suppose that $\eta_{1},\dots,\eta_{k-1}$ have been chosen in such a way that 
\begin{equation*}
    W(\ell_{j+N_{m}^{j}})\ge 4m\frac{N_{m+1}^{j}-N_{m}^{j}}{\tau_{m}^{j}}W(\ell_{j})\qquad \forall j\ge 0,\,m\ge 1:\,\, j,j+N_{m}^{j}\in \{0,\dots,k-1\}.
\end{equation*}
Note that this condition vacuously holds for $k=1$. We will prove that $\eta_{k}\ge 1$ can be chosen in such a way that the same condition holds replacing $k$ with $k+1$. This amounts to checking a finite number of conditions, namely:
\begin{equation*}
    W(\ell_{k})\ge 4m\frac{N_{m+1}^{j}-N_{m}^{j}}{\tau_{m}^{j}}W(\ell_{j})\qquad \forall j\in \{0,\dots,k-1\},\,  m\ge 1:\,\, j+N_{m}^{j}=k.
\end{equation*}
Since $\ell_{k}=\ell_{k-1}2^{-1-\eta_{k}}$, and $W(r)\to \infty$ as $r\to 0^{+}$, these finitely many constraints can be satisfied simultaneously, provided that we choose $\eta_{k}$ sufficiently large.  
\end{proof}
\noindent From this point onward, the sequence $\{\eta_{k}\}_{k\ge 1}$ refers to the one in \Cref{lem:choice-parameters}. 

For every $k\ge 0$, we introduce two time sequences $\{\hat{t}_{n}^{k,s}\}_{n\ge 2}, \{\hat{t}_{n}^{k,f}\}_{n\ge 2}\subset [0,1)$ representing, respectively, the initial and final time of the splitting process of cubes with length $\ell_{k + n-1}$ that leads to the creation of cubes with side length $\ell_{k+n}$ (see the overview in \Cref{sec: overview L1}):
\begin{equation*}
\begin{gathered}
    \hat{t}_{n}^{k,s}:= \sum_{j=1}^{m-1}2\tau_{j}^{k}+ \frac{n-1-N_{m}^{k}}{N_{m+1}^{k}-N_{m}^{k}}\tau_{m}^{k},\qquad \hat{t}_{n}^{k,f}:= \hat{t}_{n}^{k,s} + \frac{\tau_{m}^{k}}{N_{m+1}^{k}-N_{m}^{k}} \\[5pt]
    \forall n\in \{N_{m}^{k}+1,\dots,N_{m+1}^{k}\},\quad \forall m\ge 1,\quad \forall k\ge 0.
\end{gathered}
\end{equation*}
We also introduce, for every $k\ge 0$, sequences $\{t_{n}^{k,i}\}_{n\ge 2}, \{t_{n}^{k,f}\}_{n\ge 2}\subset [0,1)$ representing, respectively, the initial and final time of the translation of cubes of length $\ell_{n+k}$:

\begin{equation*}
    \begin{gathered}
    \begin{aligned}
        t_{n}^{k,s}&:= \sum_{j=1}^{m-1}2\tau_{j}^{k}+\tau_{m}^{k}+\frac{1}{T^{k}}\sum_{h=n+1}^{N_{m+1}^{k}}\tilde{\Omega}(\ell_{k+h-2}2^{-3})-\tilde{\Omega}(\ell_{k+h-1}2^{-1}),\\
        t_{n}^{k,f}&:= t_{N^k_m+1}^{k,s} +\frac{1}{T^{k}}\left(\tilde{\Omega}(\ell_{k} 2^{-N_{m}^{k}-1})-\tilde{\Omega}(\ell_{k+N_{m}^{k}}2^{-1})\right)+\frac{1}{T^{k}}\sum_{h=N_{m}^{k}+2}^{n}\tilde{\Omega}(\ell_{k} 2^{-h})-\tilde{\Omega}(\ell_{k} 2^{-h-1}),
        \end{aligned}\\[5pt]
        \forall n\in \{N_{m}^{k}+1,\dots,N_{m+1}^{k}\},\quad \forall m\ge 1,\quad  \forall k\ge 0.
    \end{gathered}
\end{equation*}

Note that $\hat{t}_{n}^{k,s}\le\hat{t}_{n}^{k,f}\le t_{n}^{k,s}\le t_{n}^{k,f}$ for all $n\ge 2$, and $\hat{t}_{n}^{k,s}, \hat{t}_{n}^{k,f},t_{n}^{k,s},t_{n}^{k,f}\to 1^{-}$ as $n\to \infty$. The splitting intervals $\{[\hat{t}_{n}^{k,s},\hat{t}_{n}^{k,f}]\}_{n\ge 2}$ have pairwise disjoint interiors, while the translation intervals $\{[t_{n}^{k,i},t_{n}^{k,f}]\}_{n\ge 2}$ are nested:
\begin{equation*}
    [t_{n}^{k,s},t_{n}^{k,f}]\subset [t_{n'}^{k,s},t_{n'}^{k,f}]\qquad \forall N_{m}^{k}+1\le n< n'\le N_{m+1}^{k},\quad \forall m\ge 1,\quad \forall k\ge 0.
\end{equation*}

The lengths of the translation intervals satisfy the following identities:
\begin{gather}
    T^{k}\left(t_{N_{m}^{k}+1}^{k,f}-t_{N_{m}^{k}+1}^{k,s}\right)= \tilde{\Omega}(\ell_{k} 2^{-N_{m}^{k}-1})-\tilde{\Omega}(\ell_{k+N_{m}^{k}}2^{-1})\qquad \forall m\ge 1,\quad \forall k\ge 0, \label{eq:length-intervals-special}\\[5pt]
    \begin{gathered}
        T^{k}\left(t_{n-1}^{k,s}-t_{n}^{k,s}\right)=\tilde{\Omega}(\ell_{k+n-2}2^{-3})-\tilde{\Omega}(\ell_{k+n-1}2^{-1}),\qquad T^{k}\left(t_{n}^{k,f}-t_{n-1}^{k,f}\right)=\tilde{\Omega}(\ell_{k} 2^{-n})-\tilde{\Omega}(\ell_{k} 2^{-n-1})\\[5pt]
        \forall n\in \{N_{m}^{k}+2,\dots,N_{m+1}^{k}\},\quad \forall m\ge 1,\quad \forall k\ge 0. 
    \end{gathered}\label{eq:length-intervals-nonspecial}
\end{gather}
\subsubsection{Flow design} 
\paragraph{Time cut-off functions.} We introduce suitable smooth time cutoff functions adapted to the translation steps of our construction. Let $\chi:\mathbb{R} \to [0, 1]$ be a non-decreasing $C^\infty$ function such that
\begin{equation}\label{eq:properties-cutoff-chi}
    \chi(t)=0\quad \forall t\le 0,\qquad \chi(t)=1\quad \forall t\ge 1,\qquad |\dot\chi(t)|\le 2\mathbbm{1}_{(0,1)}(t)\quad \forall t\in \R.
\end{equation}
We introduce smooth cut-off functions corresponding to the translation steps of the $n$th generation. In the case $n\in \{N_{m}^{k}+1\}_{m\ge 1}$, we define smooth cutoff functions $\chi_{N_{m}^{k}+1}^{k}$ such that
\begin{equation*}
    \chi_{N_{m}^{k}+1}^{k}(t):=T^{k}(t_{N_{m}^{k}+1}^{k,f}-t_{N_{m}^{k}+1}^{k,s})\chi\left(\frac{t-t_{N_{m}^{k}+1}^{k,s}}{t_{N_{m}^{k}+1}^{k,f}-t_{N_{m}^{k}+1}^{k,s}}\right)\qquad \forall m\ge 1,\quad \forall k\ge 0.
\end{equation*}
Thanks to \eqref{eq:properties-cutoff-chi} and \eqref{eq:length-intervals-special}, we have
\begin{equation}\label{eq:properties-cutoff-chin-special}
\begin{gathered}
    \chi_{N_{m}^{k}+1}^{k}(t)=0\quad \forall t\le t_{N_{m}^{k}+1}^{k, s},\qquad \chi_{N_{m}^{k}+1}^{k}(t)= \tilde{\Omega}(\ell_{k} 2^{-N_{m}^{k}-1})-\tilde{\Omega}(\ell_{k+N_{m}^{k}}2^{-1})\quad \forall t\ge t_{N_{m}^{k}+1}^{k, f},\\[5pt]
    |\dot \chi_{N_{m}^{k}+1}^{k}(t)|\le 2T^{k}\mathbbm{1}_{(t_{N_{m}^{k}+1}^{k, s},t_{N_{m}^{k}+1}^{k, f})}(t)\quad \forall t\in \R.
\end{gathered}
\end{equation}
In the case $n\not \in\{N_{m}^{k}+1\}_{m\ge 1}$, we define two smooth cutoff functions $\chi_{n}^{k,s}, \chi_{n}^{k,f}$ as follows:
\begin{equation*}
    \begin{gathered}
        \chi_{n}^{k,s}(t):=T^{k}(t_{n-1}^{k,s}-t_{n}^{k,s})\chi\left(\frac{t-t_{n}^{k,s}}{t_{n-1}^{k,s}-t_{n}^{k,s}}\right),\qquad \chi_{n}^{k,f}(t):=T^{k}(t_{n}^{k,f}-t_{n-1}^{k,f})\chi\left(\frac{t-t_{n-1}^{k,f}}{t_{n}^{k,f}-t_{n-1}^{k,f}}\right)\\[5pt]
        \forall n\in \{N_{m}^{k}+2,\dots, N_{m+1}^{k}\},\quad \forall m\ge 1,\quad \forall k\ge 0.
    \end{gathered}
\end{equation*}
As before, thanks to \eqref{eq:properties-cutoff-chi} and \eqref{eq:length-intervals-nonspecial}, we have
\begin{gather}
    \begin{gathered}
        \chi_{n}^{k,s}(t)=0\quad \forall t\le t_{n}^{k,s},\qquad \chi_{n}^{k,s}(t)=\tilde{\Omega}(\ell_{k+n-2}2^{-3})-\tilde{\Omega}(\ell_{k+n-1}2^{-1})\quad \forall t\ge t_{n-1}^{k,s},\\[5pt]
        |\dot \chi_{n}^{k,s}(t)|\le 2T^{k}\mathbbm{1}_{(t_{n}^{k,s},t_{n-1}^{k,s})}(t)\quad \forall t\in \R.
    \end{gathered}\label{eq:properties-cutoff-chin-nonspecial-in}\\[10pt]
    \begin{gathered}
        \chi_{n}^{k,f}(t)=0\quad \forall t\le t_{n-1}^{k,f},\qquad \chi_{n}^{k,f}(t)=\tilde{\Omega}(\ell_{k}2^{-n})-\tilde{\Omega}(\ell_{k}2^{-n-1})\quad \forall t\ge t_{n}^{k,f},\\[5pt]
        |\dot \chi_{n}^{k,f}(t)|\le 2T^{k}\mathbbm{1}_{(t_{n-1}^{k,f},t_{n}^{k,f})}(t)\quad \forall t\in \R.
    \end{gathered}\label{eq:properties-cutoff-chin-nonspecial-fin}
\end{gather}
\paragraph{Motion of cubes.}
Below, we prescribe the distance $r_{n}^{k}(t)$ between adjacent fictitious cubes of the $n$th generation at time $t$:
\begin{equation}\label{eq:shape-r{k}{n}}
    \begin{gathered}
        r_{1}^{k}(t):= \ell_{k} 2^{-1}\qquad \forall t\in [0,1),\quad \forall k\ge 0,\\[5pt]
        r_{N_{m}^{k}+1}^{k}(t):=
            \tilde{\Omega}^{-1}\left(\tilde{\Omega}(\ell_{k+N_{m}^{k}} 2^{-1})+\chi^{k}_{N_{m}^{k}+1}(t)\right)\qquad \forall m\ge 1,\quad \forall k\ge 0,\\[5pt]
        \begin{gathered}
            r_{n}^{k}(t):= \begin{cases}
            \tilde{\Omega}^{-1}\left(\tilde{\Omega}(\ell_{k+n-1}2^{-1})+\chi_{n}^{k,s}(t)\right)\qquad &\forall t\in [0,t_{n-1}^{k,s}),\\
            r_{n-1}^{k}(t)/4\qquad &\forall t\in [t_{n-1}^{k,s},t_{n-1}^{k,f}),\\
            \tilde{\Omega}^{-1}\left(\tilde{\Omega}(\ell_{k} 2^{-n-1})+\chi_{n}^{k,f}(t)\right)\qquad &\forall t\in [t_{n-1}^{k,f},1],
        \end{cases}\\[5pt]
        \forall n\in \{N_{m}^{k}+2,\dots,N_{m+1}^{k}\},\quad \forall m\ge 1,\quad \forall k\ge 0.
        \end{gathered}
    \end{gathered}
\end{equation}
    Using \eqref{eq:properties-cutoff-chin-special},\eqref{eq:properties-cutoff-chin-nonspecial-in}, \eqref{eq:properties-cutoff-chin-nonspecial-fin}, and the smoothness of the cutoff functions, one can verify that \begin{gather}\label{eq:endpoint-r}
        r_{n}^{k}\in C^{1}([0, 1]), \qquad r_{n}^{k}(t)=\ell_{k+n-1}2^{-1}\quad \forall t\le t_{n}^{k,s},\qquad r_{n}^{k}(t)=\ell_{k}2^{-n}\quad \forall t\ge t_{n}^{k,f},\qquad \forall n\ge 1,\quad \forall k\ge 0.
    \end{gather}
Moreover, we have
\begin{align} \label{eq: ratio dist cubes L1}
r^k_n(t) \leq \frac{r^k_{n-1}(t)}{4}, \qquad \forall t \leq t^{k, f}_{n-1},\quad \forall n \geq 2.
\end{align}
From \eqref{eq:endpoint-r}, one sees that
    \begin{align} \label{eq: supp r dot}
    \dot r_{n}^{k}=0 \qquad \forall t \in[0,1]\setminus (t_{n}^{k,s},t_{n}^{k,f}).
    \end{align}
    In addition, the following recursive formulas hold for time derivatives:
    \begin{equation}\label{eq:shape-velocity-r{k}{n}}
        \begin{gathered}
        \dot r_{1}^{k}(t)=0\qquad \forall t\in [0,1],\quad \forall k\ge 0,\\[5pt]
            \dot r_{N_{m}^{k}+1}^{k}(t)= \dot \chi_{N_{m}^{k}+1}^{k}(t)\tilde\omega(r_{N_{m}^{k}+1}^{k}(t))\qquad \forall t\in [0,1],\quad \forall m\ge 1,\quad \forall k\ge 0,\\[5pt]
            \begin{gathered}
                 \dot r_{n}^{k}(t)=\begin{cases}
            \dot\chi_{n}^{k,s}(t)\tilde\omega(r_{n}^{k}(t))\qquad &\forall t\in [0,t_{n-1}^{k,s}),\\
            \dot r_{n-1}^{k}(t)/4\qquad &\forall t\in [t_{n-1}^{k,s},t_{n-1}^{k,f}),\\
            \dot\chi_{n}^{k,f}(t)\tilde\omega(r_{n}^{k}(t))\qquad &\forall t\in [t_{n-1}^{k,f},1],
        \end{cases}\\[5pt]
        \forall n\in \{N_{m}^{k}+2,\dots,N_{m+1}^{k}\},\quad \forall m\ge 1,\quad \forall k\ge 0.
            \end{gathered}
        \end{gathered}
    \end{equation}
    
    Next, for every $k\ge 0$, $n\ge 1$, and $\sigma \in \mathfrak{S}^{n}$, we denote by $c_{n,\sigma}^{k}(t)$ the position at time $t$ of the center of the $n$th generation fictitious cube corresponding to the symbol $\sigma$:
    \begin{equation}\label{def: center of cubes L1}
        c_{n,\sigma}^{k}(t):= \sum_{j=1}^{n}\frac{r_{j}^{k}(t)}{2}\sigma_{j}\qquad \forall t\in [0,1],\quad \forall \sigma\in \mathfrak{S}^{n},\quad \forall n\ge 1,\quad \forall k\ge 0.
    \end{equation}
    From the construction we see that cubes of the same generation are essentially disjoint at each time:
    \begin{equation}\label{eq:disjoint-supports-same-generation-densities}
    \begin{gathered}
    \inter  \left(c_{n,\sigma}^{k}(t)+r_{n}^{k}(t)Q\right)\cap \inter\left(c_{n,\tilde\sigma}^{k}(t)+r_{n}^{k}(t)Q\right)=\emptyset,\\[5pt]
     \forall t\in [0,1],\quad \forall \sigma, \tilde\sigma \in \mathfrak{S}^{n},\, \sigma \neq \tilde{\sigma},\quad \forall n\ge 1,\quad \forall k\ge 0.
    \end{gathered}
    \end{equation}
Moreover, the definition \eqref{def: center of cubes L1} and \eqref{eq: ratio dist cubes L1} give
\begin{align} \label{eq: containment cubes L1}
c^k_{n, \sigma}(t) + r^k_n(t) Q \subset c^k_{n-1, \sigma^\prime}(t) + \frac{r^k_{n-1}(t)}{2} Q \qquad \forall t \in[0, t^{k,f}_{n-1}],\quad \forall  \sigma \in \mathfrak{S}^n,\quad \forall n \geq 2,
\end{align}
where we recall the notation $\sigma':=(\sigma_{1},\dots,\sigma_{n-1})\in \mathfrak{S}^{n-1}$, for all $\sigma=(\sigma_{1},\dots,\sigma_{n-1},\sigma_{n})\in \mathfrak{S}^{n}$, $n\ge 2$.

\subsection{Velocity field construction}\label{subsec:velocity-field-construction}
\subsubsection{Setting of the fixed point argument for the velocity field}
\paragraph{Space of velocity fields $\mathcal{B}$.}
We introduce a vector space of families of incompressible velocity fields: 
\begin{equation} \label{def: space B}
    \mathcal{B}:= \left\{B=\{B^{k}\}_{k\ge 0}\subset C([0,1]; C^{\omega}(\R^{d})): \text{\eqref{eq:B1} and \eqref{eq:B2} hold, and $\lVert B\rVert_{\mathcal{B}}<\infty$}\right\},
\end{equation}
where 
\begin{subequations}
    \begin{equation}\tag{B-1}\label{eq:B1}
        \diver B_{t}^{k}=0,\qquad B_{t}^{k}(x)=0\quad \forall x\in \R^{d}\setminus \ell_{k} Q\qquad \forall t\in [0,1],\quad \forall k\ge 0.
    \end{equation}
    \begin{equation}\tag{B-2}\label{eq:B2}
        B_{0}^{k}=B_{1}^{k}=0\qquad \forall k\ge 0,
    \end{equation}
\end{subequations}
and the norm $\lVert \cdot \rVert_{\mathcal{B}}$ is defined as
\begin{equation*}
    \lVert B\rVert_{\mathcal{B}}:= \sup_{k\ge 0}\sup_{t\in [0,1]}W(\ell_{k})\lVert B^{k}_{t}\rVert_{C^{\omega}(\R^{d})},
\end{equation*}
where $W$ is the weight as defined in \eqref{def: weight}. One can check that $(\mathcal{B},\lVert\cdot\rVert_{\mathcal{B}})$ is a Banach space.

\paragraph{Iteration map $\mathcal{I}_{\mathcal{B}}$.} Recall the notation $u^{e}$ for the building block introduced in \Cref{lem-Building-block}. For every $B=\{B^{k}\}_{k\ge 0}\in \mathcal{B}$, we define a sequence of vector fields $\mathcal{I}_{\mathcal{B}}(B):=\{\mathcal{I}_{\mathcal{B}}(B)^{k}\}_{k\ge 0}$ as follows:
\begin{equation} \label{def: IBB}
\begin{gathered}
    \mathcal{I}_{\mathcal{B}}(B)^{k}:[0,1]\times \R^{d}\to \R^{d}\qquad \forall k\ge 0,\\[10pt]
    \mathcal{I}_{\mathcal{B}}(B)^{k}_{t}(x):=\begin{cases}
        \underset{\sigma \in \mathfrak{S}^{n-1}}{\sum}\frac{1}{\hat{t}_{n}^{k,f}-\hat{t}_{n}^{k,s}}B^{k+n-1}\left(\frac{\hat{t}_{n}^{k,f}-t}{\hat{t}_{n}^{k,f}-\hat{t}_{n}^{k,s}},x-c_{n-1,\sigma}^{k}(\hat{t}_{n}^{k,s})\right)\qquad & \forall t\in [\hat{t}^{k,s}_{n},\hat{t}^{k,f}_{n}),\quad \forall n\ge 2,\\[10pt]
        \overunderset{N_{m+1}^{k}}{n=N_{m}^{k}+1}{\sum} \; \underset{\sigma \in \mathfrak{S}^{n}}{\sum}\frac{\dot r_{n}^{k}(t)}{2}u^{\sigma_{n}}\left(\frac{x-c_{n,\sigma}^{k}(t)}{r_{n}^{k}(t)}\right)\qquad &\forall t\in [t_{N_{m+1}^{k}}^{k,s},t_{N_{m+1}^{k}}^{k,f}),\quad \forall m\ge 1,\\[10pt]
        0\qquad &t=1.
    \end{cases}
\end{gathered}
\end{equation}
In the next subsection, we will show that $\mathcal{I}_{\mathcal{B}}$ is well-defined as a map from $\mathcal{B}$ to itself, and is a contraction with respect to the distance induced by the norm $\lVert \cdot\rVert_{\mathcal{B}}$ defined above.  

\subsubsection{Properties of the iteration map $\mathcal{I}_{\mathcal{B}}$}

\begin{prop}\label{prop:properties-iteration-map-field}
    $\mathcal{I}_{\mathcal{B}}(B)\in \mathcal{B}$ for every $B\in \mathcal{B}$. Moreover, $\mathcal{I}_{B}:\mathcal{B}\to \mathcal{B}$ is a contraction with respect to the distance induced by the norm $\lVert \cdot\rVert_{\mathcal{B}}$. In particular, there exists a unique $\bar B \in \mathcal{B}$ such that $\mathcal{I}_{\mathcal{B}}(\bar B)=\bar B$.
\end{prop}
\begin{proof}
    Let $B=\{B^{k}\}_{k\ge 0}\in \mathcal{B}$, and let us consider $\mathcal{I}_{\mathcal{B}}(B)=\{\mathcal{I}_{\mathcal{B}}(B)^{k}\}_{k\ge 0}$ as defined above. 

    \smallskip
    \noindent \textbf{Step 1: $\bm{\mathcal{I}_{\mathcal{B}}(B)}$ satisfies \eqref{eq:B1} and \eqref{eq:B2}.}
    First note that, by the definition of $\mathcal{I}_{\mathcal{B}}$, and since $B\in \mathcal{B}$,
    \begin{equation*}
        \mathcal{I}_{\mathcal{B}}(B)^{k}_{0}= \sum_{\sigma \in \mathfrak{S}}\frac{1}{\hat{t}^{k,f}_{2}}B^{k+1}(1, \cdot-c^{k}_{1,\sigma}(0))=0,\qquad  \mathcal{I}_{\mathcal{B}}(B)^{k}_{1}=0.
    \end{equation*}
Therefore, condition \eqref{eq:B2} holds for $\mathcal{I}_{\mathcal{B}}(B)$.

    Using the fact that $B\in \mathcal{B}$, along with the properties of the building blocks from \Cref{lem-Building-block}, one can check that for each $t\in [0,1]$, $\mathcal{I}_{\mathcal{B}}(B)^{k}_{t}$ is a finite sum of divergence-free and compactly supported fields in $C^{\omega}(\R^{d})$. Thus, $\mathcal{I}_{\mathcal{B}}(B)^{k}_{t}\in C^{\omega}(\R^{d})$ is compactly supported and divergence-free, for every $t\in [0,1]$. Using $\supp B^{k+n-1}\subseteq \ell_{k+n-1}Q\subseteq \ell_{k}2^{-n+1}Q$, \Cref{lem-Building-block}, the bound $r_{n}^{k}\le \ell_{k}2^{-n}$ and the definition of $c_{n-1,\sigma}^{k}$, one can show that $\supp \mathcal{I}_{\mathcal{B}}(B)^{k}_{t}\subseteq \ell_{k}Q$ for every $t\in [0,1]$. Hence, \eqref{eq:B1} holds for $\mathcal{I}_{\mathcal{B}}(B)$ as well.

    \smallskip
    \noindent \textbf{Step 2: $\bm{\mathcal{I}_{\mathcal{B}}(B)^{k}\in C([0,1];C^{\omega}(\R^{d}))}$.} Let us fix $k\ge 0$. First note that $t\mapsto \mathcal{I}_{\mathcal{B}}(B)^{k}_{t}\in C^{\omega}(\R^{d})$ is continuous in each interval $[\hat{t}_{n}^{k,s},\hat{t}_{n}^{k,f})$, $n\ge 2$, because of $B\in \mathcal{B}$. Moreover, $t\mapsto \mathcal{I}_{\mathcal{B}}(B)^{k}_{t}\in C^{\omega}(\R^{d})$ is continuous in each interval $[t_{N_{m+1}^{k}}^{k,s},t_{N_{m+1}^{k}}^{k,f})$, $m\ge 1$, by the regularity of the functions $r_{n}^{k}, c_{n,\sigma}^{k}$, and of the building blocks $u^{e}$ from \Cref{lem-Building-block}. 
    As $\mathcal{I}_{\mathcal{B}}(B)^{k}_{t}\equiv 0$ for each $t\in \{\hat{t}_{n}^{k,s}\}_{n\ge 2}\cup \{\hat{t}_{n}^{k,f}\}_{n\ge 2}\cup \{t_{N_{m+1}^{k}}^{k,s}\}_{m\ge 1}\cup \{t_{N_{m+1}^{k}}^{k,f}\}_{m\ge 1}$ because of Step 1, \eqref{eq: supp r dot} and the definition \eqref{def: IBB}, we get $\mathcal{I}_{\mathcal{B}}(B)^{k}\in C([0,1);C^{\omega}(\R^{d}))$.
    Thus, to conclude this step, it only remains to prove that $\lVert \mathcal{I}_{\mathcal{B}}(B)^{k}_{t}\rVert_{C^{\omega}(\R^{d})}\to 0$ as $t\to 1^{-}$.
    
    We first consider the case $t\in [\hat{t}^{k,s}_{n},\hat{t}^{k,f}_{n})$ for some $n\in \{N_{m}^{k}+1,\dots, N_{m+1}^{k}\}$. In the time interval $[\hat{t}^{k,s}_{n},\hat{t}^{k,f}_{n})$, $\mathcal{I}_{\mathcal{B}}(B)^{k}$ is the sum of a finite number of copies of $B^{k+n-1}$ with disjoint spatial support. Also, $\hat{t}^{k,f}_{n}-\hat{t}^{k,s}_{n}=\tau_{m}^{k}/(N_{m+1}^{k}-N_{m}^{k})$. Therefore:
    \begin{align*}
        \sup_{t\in [\hat{t}^{k,s}_{n},\hat{t}^{k,f}_{n})}\lVert \mathcal{I}_{\mathcal{B}}(B)^{k}_{t}\rVert_{C^{\omega}(\R^{d})}\le 2 \frac{N_{m+1}^{k}-N_{m}^{k}}{\tau_{m}^{k}}\sup_{t\in [0,1]}\lVert B^{k+n-1}_{t}\rVert_{C^{\omega}(\R^{d})}.
    \end{align*}
    By the definition of $\lVert \cdot\rVert_{\mathcal{B}}$, $W(\ell_{k+n-1})\ge W(\ell_{k+N_{m}^{k}})$, and \eqref{eq:choice-parameters-fixed-point}, we then get
    \begin{equation}\label{eq:bound-velocity-splitting-part}
        \begin{aligned}
        \sup_{t\in [\hat{t}^{k,s}_{n},\hat{t}^{k,f}_{n})}\lVert \mathcal{I}_{\mathcal{B}}(B)^{k}_{t}\rVert_{C^{\omega}(\R^{d})} & \le
        2 \frac{N_{m+1}^{k}-N_{m}^{k}}{\tau_{m}^{k} } \frac{1}{W(\ell_{k+n-1})}\sup_{t\in [0,1]} W(\ell_{k+n-1}) \lVert B^{k+n-1}_{t}\rVert_{C^{\omega}(\R^{d})} \\ 
        & \leq 2 \frac{N_{m+1}^{k}-N_{m}^{k}}{\tau_{m}^{k} } \frac{1}{W(\ell_{k+N_{m}^{k}})} 
        \lVert B\rVert_{\mathcal{B}} \\
        & \leq \frac{1}{2 m W(\ell_k)} 
        \lVert B\rVert_{\mathcal{B}}\qquad \forall n\in \{N_{m}^{k}+1,\dots,N_{m+1}^{k}\},\quad \forall m\ge 1.
    \end{aligned}
    \end{equation}
    
   We now consider the case $t\in [t_{N_{m+1}^{k}}^{k,s},t_{N_{m+1}^{k}}^{k,f})$. Let $\bar n \in \{N_{m}^{k}+1,\dots, N_{m+1}^{k}\}$ be such that $t\in [t_{\bar n}^{k,s},t_{\bar n}^{k,f})\setminus [t_{\bar n-1}^{k,s},t_{\bar n-1}^{k,f})$. From \eqref{eq:properties-cutoff-chin-nonspecial-in}, \eqref{eq:properties-cutoff-chin-nonspecial-fin} and \eqref{eq:shape-velocity-r{k}{n}}, we have 
    \begin{equation*}
    \begin{gathered}
        \dot r_{n}^{k}(t)=0\qquad \forall n\in\{N_{m}^{k}+1,\dots, \bar n -1\},\\[5pt]
        r_{n}^{k}(t)=4^{\bar n-n}r_{n}^{k}(t),\quad \dot r_{n}^{k}(t)= 4^{\bar n-n}\dot r_{\bar n}^{k}(t)\le 2T^{k} 4^{\bar n-n}\tilde{\omega}(r_{\bar n}^{k}(t))\qquad \forall n\in \{\bar n,\dots, N_{m+1}^{k}\}.
    \end{gathered}
    \end{equation*}
    Using \eqref{eq:disjoint-supports-same-generation-densities}, \Cref{lem-Building-block}, and the definition \eqref{def: IBB}, for each $x\in \R^{d}$ and each $n\in \{\bar n,\dots, N_{m+1}^{k}\}$, there is at most one $\sigma\in \mathfrak{S}^{n}$ for which $u^{\sigma_{n}}\left(\frac{x-c_{n,\sigma}^{k}(t)}{r_{n}^{k}(t)}\right)\neq 0$. Therefore, 
    \begin{equation}\label{eq:Linfty-bound-iterated-field-translation}
        \lVert \mathcal{I}_{\mathcal{B}}(B)^{k}_{t}\rVert_{L^{\infty}(\R^{d})}\lesssim_{d}\sum_{n=\bar n}^{N_{m+1}^{k}}\dot r_{n}^{k}(t)=2T^{k}\tilde{\omega}(r_{\bar n}^{k}(t))\sum_{n=\bar n}^{N_{m+1}^{k}}4^{\bar n-n}\lesssim T^{k}\tilde{\omega}(r_{\bar n}^{k}(t))\le \frac{T^{k}}{W(r_{\bar n}^{k}(t))}\omega(r_{\bar n}^{k}(t)).
    \end{equation}
    Similarly, taking into account \eqref{eq:disjoint-supports-same-generation-densities}, \eqref{eq: containment cubes L1}, and \Cref{lem-Building-block}, we see that for each $x\in \R^{d}$, there is at most one choice of $n\in\{\bar n,\dots, N_{m+1}^{k}\}$ and $\sigma\in \mathfrak{S}^{n}$ for which $\nabla u^{\sigma_{n}}\left(\frac{x-c_{n,\sigma}^{k}(t)}{r_{n}^{k}(t)}\right)\neq 0$. Hence, 
    \begin{equation}\label{eq:Lip-bound-iterated-field-translation}
        \lVert \nabla \mathcal{I}_{\mathcal{B}}(B)^{k}_{t}\rVert_{L^{\infty}(\R^{d})}\lesssim_{d} \sup_{\bar n\le n\le N_{m+1}^{k}}\frac{\dot r_{n}^{k}(t)}{r_{n}^{k}(t)}\lesssim T^{k}\frac{\tilde{\omega}(r_{\bar n}^{k}(t))}{r_{\bar n}^{k}(t)}\le \frac{T^{k}}{W(r_{\bar n}^{k}(t))}\frac{\omega(r_{\bar n}^{k}(t))}{r_{\bar n}^{k}(t)}.
    \end{equation}
    By \Cref{lem:interpolation}, the bounds \eqref{eq:Linfty-bound-iterated-field-translation} and \eqref{eq:Lip-bound-iterated-field-translation} together give $[\mathcal{I}_{\mathcal{B}}(B)^{k}_{t}]_{C^{\omega}(\R^{d})}\lesssim_{d} \frac{T^{k}}{W(r_{\bar n}^{k}(t))}$. Since $W(r_{\bar n}^{k}(t))\ge W(\ell_{k}2^{-N_{m}^{k}-1})$, we obtain
    \begin{equation}\label{eq:bound-velocity-translation-part}
        \sup_{t\in [t_{N_{m+1}^{k}}^{k,s},t_{N_{m+1}^{k}}^{k,f})}\lVert \mathcal{I}_{\mathcal{B}}(B)^{k}_{t}\rVert_{C^{\omega}(\R^{d})}\lesssim_{d}\frac{T^{k}}{W(\ell_{k}2^{-N_{m}^{k}-1})}.
    \end{equation}
    Putting together \eqref{eq:bound-velocity-splitting-part} and \eqref{eq:bound-velocity-translation-part}, we finally get
    \begin{equation*}
        \sup_{t\in [\hat{t}_{N_{m}^{k}+1}^{k,s},1]}\lVert \mathcal{I}_{\mathcal{B}}(B)^{k}_{t}\rVert_{C^{\omega}(\R^{d})}\le \max\left\{\frac{1}{2m}\frac{\lVert B\rVert_{\mathcal{B}}}{W(\ell_{k})}, C(d)\frac{T^{k}}{W(\ell_{k}2^{-N_{m}^{k}-1})}\right\}\to 0\qquad \text{as $m\to \infty$},
    \end{equation*}
which concludes this step.

\smallskip
\noindent \textbf{Step 3: $\bm{\mathcal{I}_{\mathcal{B}}:(\mathcal{B,\lVert \cdot \rVert_{\mathcal{B}}})\to (\mathcal{B},\lVert \cdot\rVert_{\mathcal{B}})}$ is a contraction.} First, we need to prove that $\lVert \mathcal{I}_{\mathcal{B}}(B)\rVert_{\mathcal{B}}<\infty$ for all $B\in \mathcal{B}$. By \eqref{eq:bound-velocity-splitting-part}, \eqref{eq:bound-velocity-translation-part}, the monotonicity of $W$ and the bound on $T^{k}$ from \eqref{eq:def-Tk}, we get
\begin{align*}
    \lVert \mathcal{I}_{\mathcal{B}}(B)\rVert_{\mathcal{B}}&=\sup_{k\ge 0}\sup_{t\in [0,1]}W(\ell_{k})\lVert \mathcal{I}_{\mathcal{B}}(B)^{k}_{t}\rVert_{C^{\omega}(\R^{d})}\\
    &= \sup_{k\ge 0}\sup_{m\ge 1}\max \left\{\sup_{t\in \bigcup_{n=N_{m}^{k}+1}^{N_{m+1}^{k}}[\hat{t}^{k,s}_{n},\hat{t}^{k,f}_{n})}W(\ell_{k})\lVert \mathcal{I}_{\mathcal{B}}(B)^{k}_{t}\rVert_{C^{\omega}(\R^{d})},  \sup_{t\in [t_{N_{m+1}^{k}}^{k,s},t_{N_{m+1}^{k}}^{k,f})}W(\ell_{k})\lVert \mathcal{I}_{\mathcal{B}}(B)^{k}_{t}\rVert_{C^{\omega}(\R^{d})}\right\}\\
    &\le \sup_{k\ge 0}\sup_{m\ge 1}\max\left\{\frac{1}{2m}\lVert B\rVert_{\mathcal{B}}, C(d)T^{k}\frac{W(\ell_{k})}{W(\ell_{k}2^{-N_{m}^{k}-1})}\right\}\\
    &\le \max \left\{\frac{1}{2}\lVert B\rVert_{\mathcal{B}}, 12 C(d)\tilde{\Omega}\left(\frac{1}{4}\right)\right\}<\infty.
\end{align*}
Now let $B,\tilde{B}\in \mathcal{B}$. Note that 
$\mathcal{I}_{\mathcal{B}}(B)^{k}_{t}-\mathcal{I}_{\mathcal{B}}(\tilde{B})^{k}_{t}\equiv 0$ for every $t\in [t_{N_{m+1}^{k}}^{k,s}, t_{N_{m+1}^{k}}^{k,f})$ and every $m\ge 1$. On the other hand, 
$\mathcal{I}_{\mathcal{B}}(B)^{k}_{t}-\mathcal{I}_{\mathcal{B}}(\tilde{B})^{k}_{t}=\mathcal{I}_{\mathcal{B}}(B-\tilde{B})^{k}_{t}$ for every $t\in [\hat{t}_{n}^{k,s},\hat{t}_{n}^{k,f})$ and every $n\ge 2$. Therefore, applying \eqref{eq:bound-velocity-splitting-part} to $B-\tilde B \in \mathcal{B}$, we get 
\begin{align*}
    \lVert \mathcal{I}_{B}(B)-\mathcal{I}_{\mathcal{B}}(\tilde B)\rVert_{\mathcal{B}}= \sup_{k\ge 0}\sup_{n\ge 2}\sup_{t\in [\hat{t}^{k,s}_{n},\hat{t}^{k,f}_{n})}W(\ell_{k})\lVert \mathcal{I}_{\mathcal{B}}(B-\tilde B)^{k}_{t}\rVert_{C^{\omega}(\R^{d})}\le \frac{1}{2}\lVert B-\tilde B\rVert_{\mathcal{B}},
\end{align*}
which proves the desired contraction property and concludes the proof.
\end{proof}

\subsection{Density construction} \label{subsec: density construction}
\subsubsection{Setting of the fixed point argument for the density field}
\paragraph{Space of density fields $\mathcal{T}$.}
We fix $s > d/2 + 1$ and we consider the negative Sobolev space $H^{-s}(\R^{d})$, defined as the dual of the function space $H^{s}(\mathbb{R}^d)$, i.e. we have $f \in H^{-s}(\mathbb{R}^d)$ if
\begin{align*}
\|f\|_{H^{-s}(\mathbb{R}^d)} \coloneqq \sup_{\substack{\varphi \in H^s_0(\mathbb{R}^d) \\ \|\varphi\|_{H^s(\mathbb{R}^d)} \leq 1}} |\langle f, \varphi \rangle| < \infty,
\end{align*}
where $\langle \cdot, \cdot \rangle: H^{-s}(\mathbb{R}^d)\times H^{s}(\mathbb{R}^d) \to \mathbb{R}$ denotes the duality pairing. Let $\Lip([0,1]; H^{-s}(\R^{d}))$ be the space of continuous curves $[0,1]\ni t\mapsto \theta_{t}\in H^{-s}(\R^{d})$ for which
\begin{align*}
[\theta]_{\Lip([0,1]; H^{-s}(\R^{d}))} \coloneqq \sup_{\substack{t_1, t_2 \in [0, 1] \\ t_1 \neq t_2}}\frac{1}{|t_1-t_2|} \|\theta(t_1, \cdot) - \theta(t_2, \cdot)\|_{H^{-s}(\mathbb{R}^d)}<\infty.
\end{align*}

We define the following space of sequences of density fields:
\begin{align} \label{def: space tau}
    \mathcal{T}:=  \left\{\Theta=\{\Theta^{k}\}_{k\ge 0}\subset \Lip([0,1]; H^{-s}(\R^{d})): \text{\eqref{eq:T1} and \eqref{eq:T2} hold, and $\lVert \Theta\rVert_{\mathcal{T}}<\infty$} \right\},
\end{align}
where 
\begin{subequations}
    \begin{equation}\tag{T-1}\label{eq:T1}
        \supp \Theta^{k}_{t} \subseteq \ell_{k} Q\qquad \forall t\in [0,1], \quad \forall k\ge 0.
    \end{equation}
    \begin{equation}\tag{T-2}\label{eq:T2}
        \Theta^{k}_{0}= \frac{1}{\mathscr{L}^{d}(A^{k})}\mathbbm{1}_{A^{k}},\quad A^{k}:= \ell_{k}\bigcup_{\sigma\in \mathfrak{S}} \left( \frac{\sigma}{4}+\ell_{k+1}Q \right),\qquad \Theta^{k}_{1}=\frac{1}{\mathscr{L}^{d}(\ell_{k} Q)}\mathbbm{1}_{\ell_{k} Q}\qquad \forall k\ge 0,
    \end{equation}
\end{subequations}
and the semi-norm $[\cdot]_{\mathcal{T}}$ is defined as
\begin{align} \label{def: norm T}
 [\Theta ]_{\mathcal{T}}:= \sup_{k \ge 0} W(\ell_k) [\Theta^k]_{\Lip([0,1]; H^{-s}(\R^{d}))}.
\end{align}
To make sure that $\mathcal{T}$ is nonempty, observe that $\Theta=\{\Theta^{k}\}_{k\ge 0}$ obtained as linear interpolation
\begin{equation*}
     \Theta^{k}_{t}:= (1-t)\frac{1}{\mathscr{L}^{d}(A^{k})}\mathbbm{1}_{A^{k}}+t\frac{1}{\mathscr{L}^{d}(\ell_{k} Q)}\mathbbm{1}_{\ell_{k} Q}\qquad \forall t\in [0,1],\quad \forall k\ge 0
\end{equation*}
belongs to the space $\mathcal{T}$. In fact, for such $\Theta^{k}$ we have
\begin{equation*}
    [\Theta^{k}]_{\Lip([0,1];H^{-s}(\R^{d}))}=\lVert \Theta^{k}_{1}-\Theta^{k}_{0}\rVert_{H^{-s}(\R^{d})}\lesssim_{d,s} \ell_{k}\le \frac{\ell_{k}\omega(\ell_{k})}{\tilde{\omega}(\ell_{k})}\frac{1}{W(\ell_{k})}\le \frac{\omega(1)}{\tilde{\omega}(1)}\frac{1}{W(\ell_{k})},
\end{equation*}
where we used the definition of $W$, the monotonicity of $\omega$, and the concavity of $\tilde{\omega}$.

Finally, we equip $\mathcal{T}$ with the distance
\begin{equation*}
    d_{\mathcal{T}}(\Theta,\tilde\Theta):= [\Theta-\tilde{\Theta}]_{\mathcal{T}}= \sup_{k\ge 0} W(\ell_{k})[\Theta^k-\tilde\Theta^{k}]_{\Lip([0,1]; H^{-s}(\R^{d}))}.
\end{equation*}
One can check that $(\mathcal{T},d_{\mathcal{T}})$ is indeed a complete metric space. 

\paragraph{Iteration map $\mathcal{I}_{\mathcal{T}}$.} For every $\Theta=\{\Theta^{k}\}_{k\ge 0}$, we define a corresponding sequence $\mathcal{I}_{\mathcal{T}}(\Theta):=\{\mathcal{I}_{\mathcal{T}}(\Theta)^{k}\}_{k\ge 0}$ as follows:
    \begin{equation} \label{eq: ITT}
    \begin{gathered}
    [0,1]\ni t\mapsto \mathcal{I}_{\mathcal{T}}(\Theta)^{k}_{t}\in H^{-s}(\R^{d})\qquad \forall k\ge 0,\\[10pt]
        \mathcal{I}_{\mathcal{T}}(\Theta)^{k}_{t}:=\begin{cases}
            \underset{\sigma\in \mathfrak{S}^{n-1}}{\sum}2^{-(n-1)d}\Theta^{k+n-1}\left(\frac{\hat{t}_{n}^{k,f}-t}{\hat{t}_{n}^{k,f}-\hat{t}_{n}^{k,s}},\cdot-c_{n-1,\sigma}^{k}(\hat{t}_{n}^{k,s})\right)\qquad &\forall t\in [\hat{t}^{k,s}_{n},\hat{t}^{k,f}_{n}),\quad \forall n\ge 2,\\[10pt]
            \underset{\sigma \in \mathfrak{S}^{N_{m+1}^{k}}}{\sum}2^{-N_{m+1}^{k}d}\ell_{k+N_{m+1}^{k}}^{-d}\mathbbm{1}_{Q}\left(\frac{x-c_{N_{m+1}^{k},\sigma}^{k}(t)}{\ell_{k+N_{m+1}^{k}} }\right)\qquad & \forall t\in [t_{N_{m+1}^{k}}^{k,s},t_{N_{m+1}^{k}}^{k,f}),\quad \forall m\ge 1,\\[10pt]
            \frac{1}{\mathscr{L}^{d}(\ell_{k} Q)}\mathbbm{1}_{\ell_{k} Q}(x)\qquad & t=1.
        \end{cases}
    \end{gathered}
    \end{equation}
    For each time $t$ belonging to intervals of the type $[\hat{t}^{k,s}_{n},\hat{t}^{k,f}_{n})$, $\mathcal{I}_{\mathcal{T}}(\Theta)^{k}_{t}$ is a finite renormalized sum of elements from $\Theta=\{\Theta^{k}\}_{k\ge 0}$. In intervals of the type $[t^{k,s}_{N_{m+1}^{k}},t^{k,f}_{N_{m+1}^{k}})$ instead, $\mathcal{I}_{\mathcal{T}}(\Theta)^{k}_{t}$ is an explicit non-negative bounded function with unit mass. In the following lemma we prove that $\mathcal{I}_{\mathcal{T}}(\Theta)^{k}$ solves the continuity equation with vector field $\bar{B}^k$ in those intervals, where $\bar B=\{\bar B^{k}\}_{k\ge 0}$ is the fixed point given by \Cref{prop:properties-iteration-map-field}.

\begin{lem} \label{lem: Thetak solve continuity}
For every $\Theta \in \mathcal{T}$, $k\ge 0$, and $m\ge 1$, $\mathcal{I}_{\mathcal{T}}(\Theta)^k$ solves the continuity equation with velocity field $\bar{B}^k$ in the time interval $(t_{N_{m+1}^{k}}^{k,s},t_{N_{m+1}^{k}}^{k,f})$.
\end{lem}
\begin{proof}
    From the definition \eqref{eq: ITT}, we see that for every $t\in (t_{N_{m+1}^{k}}^{k,s},t_{N_{m+1}^{k}}^{k,f})$, $\mathcal{I}_{\mathcal{T}}(\Theta)^k_{t}$ coincides with the unit mass renormalization of the characteristic function of the union of $2^{N^k_{m+1}d}$ disjoint cubes of side length $\ell_{k + N^k_{m+1}}$, centered at the points $c^k_{N^k_{m+1}, \sigma}(t)$, for $\sigma \in \mathfrak{S}^{N^k_{m+1}}$. Therefore, it suffices to prove the following:
    \begin{equation*}
        \bar B(t,x)=\dot c^k_{N^k_{m+1}, \sigma}(t)\qquad \forall t\in (t_{N_{m+1}^{k}}^{k,s},t_{N_{m+1}^{k}}^{k,f}),\quad \forall x\in c^k_{N^k_{m+1}, \sigma}(t)+\ell_{k + N^k_{m+1}}Q,\quad \forall \sigma \in  \mathfrak{S}^{N^k_{m+1}}.
    \end{equation*}
    Let $t\in (t_{N_{m+1}^{k}}^{k,s},t_{N_{m+1}^{k}}^{k,f})$ and $\sigma \in  \mathfrak{S}^{N^k_{m+1}}$ be given. For every $n\le N^k_{m+1}$, we call $\sigma^{n}\in \mathfrak{S}^{n}$ the unique symbol such that $\sigma^{n}\subset \sigma$. Also, Let $\bar n\in \{N_{m}^{k}+1,\dots, N_{m+1}^{k}\}$ be the minimum integer for which $t\in (t_{\bar n}^{k,s},t_{\bar n}^{k,f})$. Because of \eqref{eq: supp r dot}, we will have:
    \begin{equation}\label{eq:consideration-1-proof-explicit-solution}
        \dot r_{n}^{k}(t)=0\qquad \forall n\in \{N_{m}^{k}+1,\dots, \bar n-1\}.
    \end{equation}
    Moreover, by \eqref{eq: containment cubes L1} and $\ell_{k + N^k_{m+1}}\le r^k_{N^k_{m+1}}(t)/2$, the following inclusions hold:
    \begin{align}\label{eq:consideration-2-proof-explicit-solution}
c^k_{N^k_{m+1}, \sigma}(t) + \ell_{k + N^k_{m+1}}Q\subset c^k_{N^k_{m+1}, \sigma}(t) + \frac{r^k_{N^k_{m+1}}(t)}{2} Q \subset c^k_{n, \sigma^n}(t) + \frac{r^k_n(t)}{2} Q\qquad \forall n\in \{\bar n,\dots, N_{m+1}^{k}\}.
\end{align}
From \eqref{eq:disjoint-supports-same-generation-densities} and \eqref{eq:consideration-2-proof-explicit-solution} we also deduce that 
\begin{align}\label{eq:consideration-3-proof-explicit-solution}
c^k_{N^k_{m+1}, \sigma}(t) + \ell_{k + N^k_{m+1}}Q\subset \R^{d}\setminus \left(c^k_{n, \tilde\sigma}(t) + r^k_n(t)Q\right)\qquad \forall n\in \{\bar n,\dots, N_{m+1}^{k}\},\quad \forall \tilde\sigma\in \mathfrak{S}^{n},\,\tilde\sigma\neq \sigma^{n}.
\end{align}
Since $\bar B=\mathcal{I}_{\mathcal{B}}(\bar B)$, at time $t$, $\bar B^{k}_{t}$ is given by the explicit expression from \eqref{def: IBB}. Therefore, by the properties of the building blocks from \Cref{lem-Building-block}, combining \eqref{eq:consideration-1-proof-explicit-solution},\eqref{eq:consideration-2-proof-explicit-solution}, and \eqref{eq:consideration-3-proof-explicit-solution} we get
\begin{equation*}
    \bar B^{k}(t,x)= \sum_{n=\bar n}^{N_{m+1}^{k}}\frac{\dot r_{n}^{k}(t)}{2}u^{\sigma^{n}_{n}}\left(\frac{x-c_{n,\sigma^{n}}^{k}(t)}{r_{n}^{k}(t)}\right)=\sum_{n=\bar n}^{N_{m+1}^{k}}\frac{\dot r_{n}^{k}(t)}{2}\sigma_{n}=\dot c^k_{N^k_{m+1}, \sigma}(t)\qquad \forall x\in c^k_{N^k_{m+1}, \sigma}(t)+\ell_{k + N^k_{m+1}}Q,
\end{equation*}
 where the last identity follows from \eqref{def: center of cubes L1}. This concludes the proof.
\end{proof}

In the next subsection we will prove that $\mathcal{I}_{\mathcal{T}}$ is well-defined as a map from $\mathcal{T}$ to itself, and is a contraction with respect to the distance $d_{\mathcal{T}}$ defined above.

\subsubsection{Properties of the iteration map $\mathcal{I}_{\mathcal{T}}$.}
\begin{prop}\label{prop:properties-iteration-map-measure}
    $\mathcal{I}_{\mathcal{T}}(\Theta)\in \mathcal{T}$, for every $\Theta \in \mathcal{T}$. Moreover, the map $\mathcal{I}_{\mathcal{T}}:\mathcal{T}\to \mathcal{T}$ is a contraction with respect to the distance $d_{\mathcal{T}}$. In particular, there exists a unique $\bar \Theta\in \mathcal{T}$ such that $\mathcal{I}_{\mathcal{T}}(\bar \Theta)=\bar \Theta$.
\end{prop}
\begin{proof}
    Let $\Theta=\{\Theta^{k}\}_{k\ge 0}\in \mathcal{T}$, and let us consider $\mathcal{I}_{\mathcal{T}}(\Theta)=\{\mathcal{I}_{\mathcal{T}}(\Theta)^{k}\}_{k\ge 0}$ as above.

    \smallskip
    \noindent \textbf{Step 1: $\bm{\mathcal{I}_{\mathcal{T}}(\Theta)}$ satisfies \eqref{eq:T1} and \eqref{eq:T2}.} Since $t_{2}^{k,s}=0$, $c_{1,\sigma}^{k}(0)= \ell_{k}\sigma/4$ for every $\sigma \in \mathfrak{S}$, and $\Theta^{k+1}(1,\cdot)= \mathscr{L}^{d}(\ell_{k+1}Q)^{-1}\mathbbm{1}_{\ell_{k+1}Q}(x)$, at time $t=0$, we have
\begin{equation*}
    \mathcal{I}_{\mathcal{T}}(\Theta)^{k}_{0}= \sum_{\sigma \in \mathfrak{S}}2^{-d}\mathscr{L}^{d}(\ell_{k+1}Q)^{-1}\mathbbm{1}_{\frac{\ell_{k}}{4}\sigma+\ell_{k+1}Q}.
\end{equation*}
Since at time $t=1$, $\mathcal{I}_{\mathcal{T}}(\Theta)^{k}(1,\cdot):= \mathscr{L}^{d}(\ell_{k} Q)^{-1}\mathbbm{1}_{\ell_{k} Q}$ by definition, we conclude that condition \eqref{eq:T2} holds for $\mathcal{I}_{\mathcal{T}}(\Theta)$. To prove that $\mathcal{I}_{\mathcal{T}}(\Theta)$ satisfies \eqref{eq:T1}, one can proceed exactly as in Step 1 of the proof of \Cref{prop:properties-iteration-map-field}.

\smallskip
\noindent \textbf{Step 2: $\bm{\mathcal{I}_{\mathcal{T}}(\Theta)^k \in \Lip([0, 1]; H^{-s}(\mathbb{R}^d))}$.}
First of all, because of \Cref{lem: Thetak solve continuity}, for every $\varphi \in H^{s}(\R^{d})$ the following estimate holds for any given $t\in (t^{k, s}_{N^k_{m+1}}, t^{k, f}_{N^k_{m+1}})$, $m\ge 1$:
\begin{align*}
    \left|\frac{d}{dt}\int_{\R^{d}}\varphi \,  \mathcal{I}_{\mathcal{T}}(\Theta)^{k}_{t}\, {\rm d}x\right|&=\left|\int_{\mathbb{R}^d}  \bar{B}^k_{t} \cdot \nabla \varphi\,\mathcal{I}_{\mathcal{T}}(\Theta)^{k}_{t} \, {\rm d}x \right| \\
    &\le \lVert \bar{B}^k_{t}\rVert_{L^{\infty}(\R^{d})}\lVert \nabla \varphi\rVert_{L^{\infty}(\R^{d})}\\
    &\le  T^k \tilde{\omega} (\ell_k 2^{-N^k_m - 1}) \| \varphi \|_{H^{s}(\R^{d})}\le 12 \tilde{\Omega} \left(\frac{\ell_k}{4}\right) \tilde{\omega} \left(\frac{\ell_k}{4}\right) \| \varphi \|_{H^{s}(\R^{d})}.
\end{align*}
In the third inequality, we used \eqref{eq:Linfty-bound-iterated-field-translation} and the Sobolev embedding $H^{s}(\R^{d})\hookrightarrow C^{1}(\R^{d})$, due to $s>d/2+1$. The last inequality, instead, was derived using \eqref{eq:def-Tk} and the monotonicity of $\tilde{\omega}$. 
In particular, we have the following:
\begin{equation}\label{eq:Lip-Hs-explicit}
    [ \mathcal{I}_{\mathcal{T}}(\Theta)^k ]_{\Lip\Big((t^{k, s}_{N^k_{m+1}}, t^{k, f}_{N^k_{m+1}}); H^{-s}(\mathbb{R}^d)\Big)}\le 12 \tilde{\Omega} \left(\frac{\ell_k}{4}\right) \tilde{\omega} \left(\frac{\ell_k}{4}\right)\qquad \forall m\ge 1.
\end{equation}

Next, we note that on each interval $(\hat{t}^{k, s}_n, \hat{t}^{k, f}_n)$, $n\ge 2$, we have the following estimate:
\begin{equation}\label{eq: Lip Hs implicit}
    \begin{aligned} 
[ \mathcal{I}_{\mathcal{T}}(\Theta)^k ]_{\Lip\big((\hat{t}^{k, s}_n, \hat{t}^{k, f}_n); H^{-s}(\mathbb{R}^d)\big)} & \leq \frac{1}{\hat{t}^{k, f}_n - \hat{t}^{k, s}_n} [ \Theta^{k+n-1} ]_{\Lip([0, 1]; H^{-s}(\mathbb{R}^d))} \\
 & \leq \frac{N^k_{m+1} - N^k_m}{\tau^k_m} \frac{1}{W(\ell_{k+n-1})} [ \Theta ]_{\mathcal{T}}\leq \frac{1}{4} \frac{1}{W(\ell_k)} [ \Theta ]_{\mathcal{T}},
\end{aligned}\qquad \forall n\ge 2.
\end{equation}
The first inequality follows from a rescaling and change of variables. The second inequality is derived from the definition of $\hat{t}^{k,s}_n$, $\hat{t}^{k,f}_n$ and \eqref{def: norm T}. Finally, in the last inequality, we used \Cref{lem:choice-parameters}. By the design of \eqref{eq: ITT}, for a sequence $\{\Theta^{k}\}_{k\ge 0}$ that belongs to $\mathcal{T}$, in particular, satisfying \eqref{eq:T2}, the density $\mathcal{I}_{\mathcal{T}}(\Theta)^k$ is continuous at times $\hat{t}^{k,s}_n, \hat{t}^{k,f}_n, t^{k,s}_{N^k_{m+1}}, t^{k,f}_{N^k_{m+1}}$ for $n \geq 2$ and $m \geq 1$. Combining this with estimates \eqref{eq:Lip-Hs-explicit} and \eqref{eq: Lip Hs implicit}, we obtain
\begin{align} \label{eq: Lip Hs [1,0)}
[ \mathcal{I}_{\mathcal{T}}(\Theta)^k ]_{\Lip([0, 1); H^{-s}(\mathbb{R}^d))} \leq \max \left\{ 12 \tilde{\Omega} \left(\frac{\ell_k}{4}\right) \tilde{\omega} \left(\frac{\ell_k}{4}\right), \frac{1}{4} \frac{1}{W(\ell_k)} [\Theta]_{\mathcal{T}}  \right\}.
\end{align}
Basically, we have a uniform Lipschitz estimate for any $t < 1$. In order to conclude this step, obtaining the same Lipschitz bound in the closed interval $[0,1]$, it suffices to prove that $\mathcal{I}_{\mathcal{T}}(\Theta)^k_{t}$ converges to $\mathbbm{1}_{\ell_{k} Q}/{\mathscr{L}^{d}(\ell_{k} Q)}$ in $H^{-s}(\R^{d})$ as $t \to 1^{-}$. By \eqref{eq: Lip Hs [1,0)}, we may reduce to prove the convergence on the specific time sequence $t = t^{k, f}_{N^k_{m+1}}$, where $\mathcal{I}_{\mathcal{T}}(\Theta)^k_t$ is composed of $2^{N^k_{m+1}d}$ cubes, each of side length $\ell_{k + N^k_{m+1}}$, uniformly distributed in a dyadic grid inside the cube $\ell_k Q$. Inside these cubes the magnitude of $\mathcal{I}_{\mathcal{T}}(\Theta)^k$ is uniform with value $2^{-N^k_{m+1}d} \, \ell_{k + N^k_{m+1}}^{-d}$. Therefore, $\mathcal{I}_{\mathcal{T}}(\Theta)^k_{t}$ with $t = t^{k, f}_{N^k_{m+1}}$ weakly converges to $\mathbbm{1}_{\ell_{k} Q}/{\mathscr{L}^{d}(\ell_{k} Q)}$ as $m \to \infty$.

\smallskip
\noindent {\textbf{Step 3: $\bm{\mathcal{I}_{\mathcal{T}}:(\mathcal{T},d_{\mathcal{T}}) \to (\mathcal{T},d_{\mathcal{T}})}$ is a contraction.}} First of all we need to show that $[\mathcal{I}_{\mathcal{T}}(\Theta)]_{\mathcal{T}}<\infty$ for all $\Theta \in \mathcal{T}$. In fact, using the Lipschitz bound obtained in Step 1, and the definition of $W$, we get
\begin{align*}
    [ \mathcal{I}_{\mathcal{T}}(\Theta) ]_{\mathcal{T}}&= \sup_{k\ge 0}W(\ell_{k})[ \mathcal{I}_{\mathcal{T}}(\Theta)^k ]_{\Lip([0, 1]; H^{-s}(\mathbb{R}^d))}\\
    &\le \max \left\{ \sup_{k \geq 0} 12 \tilde{\Omega} \left(\frac{\ell_k}{4}\right) \omega \left(\frac{\ell_k}{4}\right), \frac{1}{4} [\Theta ]_{\mathcal{T}}  \right\}<\infty.
\end{align*}
Let now $\Theta, \tilde{\Theta}\in \mathcal{T}$ be given. Note that $\mathcal{I}_{\mathcal{T}}(\Theta)^{k}$ coincides with $\mathcal{I}_{\mathcal{T}}(\tilde\Theta)^{k}$ in all time intervals $(t^{k, s}_{N^k_{m+1}}, t^{k, f}_{N^k_{m+1}})$, $m\ge 1$. On the other hand, on each interval $(\hat{t}^{k, s}_n, \hat{t}^{k, f}_n)$, $n\ge 2$, arguing exactly as in \eqref{eq: Lip Hs implicit} we get
\begin{equation*}
    [\mathcal{I}_{\mathcal{T}}(\Theta)^{k}-\mathcal{I}_{\mathcal{T}}(\tilde\Theta)^{k}]_{\Lip\left((\hat{t}^{k, s}_n, \hat{t}^{k, f}_n); H^{-s}(\mathbb{R}^d)\right)}\le \frac{1}{4 W(\ell_{k})}d_{\mathcal{T}}(\Theta,\tilde\Theta)\qquad \forall n\ge 2. 
\end{equation*}
Therefore, we deduce
\begin{align*}
    d_{\mathcal{T}}\left(\mathcal{I}_{\mathcal{T}}(\Theta),\mathcal{I}_{\mathcal{T}}(\tilde\Theta)\right)&=\sup_{k\ge 0} W(\ell_{k})[\mathcal{I}_{\mathcal{T}}(\Theta)^{k}-\mathcal{I}_{\mathcal{T}}(\tilde\Theta)^{k}]_{\Lip([0, 1]; H^{-s}(\mathbb{R}^d))}\le \frac{1}{4}d_{\mathcal{T}}(\Theta,\tilde\Theta),
\end{align*}
 which is the desired contraction property, and concludes the proof of the proposition.
\end{proof}

\subsection{Proof of \Cref{prop:non-uniqeness-L1}}\label{subsec:proof-nonuniqueness-L1}
Let $\bar{B}$ and $\bar{\Theta}$ be the fixed points given by \Cref{prop:properties-iteration-map-field} and \Cref{prop:properties-iteration-map-measure}, respectively. We denote by $\mathscr{P}(\R^{d})$ the space of probability measures in $\R^{d}$.

\smallskip 
\noindent \textbf{Step 1: $\bm{\bar \Theta^{k}\in L^{\infty}([0,1];L^{1}(\R^{d}))\cap C_{w^{*}}([0,1];\mathscr{P}(\R^{d}))}$ solves \eqref{eq:PDE} with velocity field $\bm{\bar{B}^{k}}$.}
For every $k\ge 0$, let $\mathcal{E}_{k}\subset (0,1)$ be the largest open set such that the following hold for every $t\in \mathcal{E}_{t}$:
\begin{gather}
    \bar\Theta^{k}_{t}\in L^{1}(\R^{d}),\qquad \lVert \bar\Theta^{k}_{t}\rVert_{L^{1}(\R^{d})}=1, \qquad \bar\Theta^{k}_{t}\ge 0,\\[5pt]
    \frac{d}{dt} \int_{\mathbb{R}^d}  \varphi \, \bar{\Theta}^k_t \, {\rm d} x= \int_{\mathbb{R}^d}  \bar{B}^k_t \cdot \nabla \varphi \,\bar{\Theta}^k_t\, {\rm d} x \qquad \forall \varphi \in C^{\infty}_{c}(\R^{d}).\label{eq: Thetak Bk continuity}
\end{gather}
As $\bar \Theta, \bar B$ are fixed points of the iteration maps $\mathcal{I}_{\mathcal{T}}, \mathcal{I}_{\mathcal{B}}$ from \eqref{eq: ITT} and \eqref{def: IBB}, respectively, taking \Cref{lem: Thetak solve continuity} into account we get
\begin{equation}\label{eq:recursion-L1-times}
    \mathscr{L}^1(\mathcal{E}_k) \geq \frac{1}{2} + \sum_{n=2}^{\infty} (\hat{t}^{k,f}_{n}- \hat{t}^{k,s}_{n}) \mathscr{L}^1(\mathcal{E}_{k+n-1}),\qquad \sum_{n=2}^{\infty} (\hat{t}^{k,f}_{n}- \hat{t}^{k,s}_{n})=\frac{1}{2}\qquad \forall k\ge 0.
\end{equation}
This is because, on the one hand, in all time intervals $(t_{N_{m+1}^{k}}^{k,s},t_{N_{m+1}^{k}}^{k,f})$, $m\ge 1$, whose lengths add up to $1/2$, $\bar\Theta^{k}=\mathcal{I}_{\mathcal{T}}(\bar\Theta)^{k}$ is explicitly given by a non-negative $L^{1}$ function with unit mass, and \eqref{eq: Thetak Bk continuity} holds in these intervals by \Cref{lem: Thetak solve continuity}. On the other hand, in the intervals $(\hat{t}^{k,s}_{n},\hat{t}^{k,f}_{n})$, $n\ge 2$, $\bar\Theta^{k}=\mathcal{I}_{\mathcal{T}}(\bar\Theta)^{k}$ and $\bar B^{k}=\mathcal{I}_{\mathcal{B}}(\bar B)^{k}$ are given by a finite sum of renormalized copies of $\bar\Theta^{k+n-1}$ and $\bar B^{k+n-1}$, reparametrized in time in an affine way. 

Let $\eps:=\inf_{k\ge 0}\mathscr{L}^{1}(\mathcal{E}_{k})\in [0,1]$. From \eqref{eq:recursion-L1-times} we deduce $\eps \ge 1/2+ \eps/2$, that is $\eps=1$. In particular, $\mathcal{E}_{k}$ is dense in $[0,1]$ for all $k\ge 0$. As a consequence, by the $H^{-s}(\R^{d})$ continuity in time of $\bar\Theta^{k}$, and the density of $H^{s}(\R^{d})$ in $C_{0}(\R^{d})$, we deduce that 
  $\bar{\Theta}^k_t \in \mathscr{P}(\mathbb{R}^d)$ for all $t \in [0, 1]$, and $\bar{\Theta}^k$ is continuous in time with respect to the weak-$*$ convergence of measures. Moreover, $\bar{\Theta}^k \in \Lip([0, 1]; H^{-s}(\mathbb{R}^d))$ combined with the fact that \eqref{eq: Thetak Bk continuity} holds almost everywhere in $(0,1)$ implies that $\bar{\Theta}^k$ solves the continuity equation with velocity field $\bar B^{k}$, concluding this step.

\smallskip
\noindent \textbf{Step 2: Construction of $\bm{\mu}$ and $\bm{b}$.}
We construct $\mu\in C_{w^{*}}([0,1];\mathscr{P}(\R^{d}))$ and $v\in C([0,1];C^{\omega}(\R^{d}))$ from $\bar\Theta^{0}, \bar B^{0}$, respectively, inverting the direction of time:
\begin{equation*}
    \mu_{t}:=\bar\Theta^{0}_{1-t},\qquad v_{t}:= -\bar B^{0}_{1-t}\qquad \forall t\in [0,1].
\end{equation*}
Since $\bar B \in \mathcal{B}$, we will have 
\begin{equation*}
    \supp v_{t}\subseteq Q,\qquad \diver v_{t}=0\qquad \forall t\in [0,1].
\end{equation*}
Similarly, since $\bar\Theta \in\mathcal{T}$, we will have
\begin{equation*}
    \supp \mu_{t}\subseteq Q\qquad \forall t\in [0,1], \qquad \mu_{0}=\mathbbm{1}_{Q}\mathscr{L}^{d},\qquad \mu_{1}= \mathscr{L}^{d}(A)^{-1}\mathbbm{1}_{A}\mathscr{L}^{d}, \quad A= \bigcup_{\sigma\in \mathfrak{S}}\frac{\sigma}{4}+\ell_{1}Q.
\end{equation*}
Finally, by Step 1 $\mu$ solves \eqref{eq:PDE} with velocity field $v$ in the time interval $[0,1]$, and $\mu_{t}\ll \mathscr{L}^{d}$ for almost every $t\in [0,1]$. This concludes the proof for $T=1$. At this point, appropriately reparametrizing in time, we can indeed conclude the construction on any time interval $[0,T]$.

\bigskip
\paragraph{Acknowledgements.}
    R.C. is supported by the Swiss National Science Foundation (SNF grant
PZ00P2\_208930), by the Swiss State Secretariat for Education, Research and Innovation (SERI) under
contract number MB22.00034.

\bibliographystyle{alpha}
\bibliography{Nonuniqueness.bib}
\end{document}